\documentclass{amsart}

\usepackage{graphicx}              
\usepackage{amsmath}               
\usepackage{amsfonts}              
\usepackage{amsthm}                
\usepackage{amssymb}
\usepackage{amscd}
 \usepackage{listings}
\usepackage{times}
\usepackage{comment}
\usepackage{mathtools}
\usepackage{tikz}
\usepackage{relsize}

\usepackage{dsfont}
\usepackage{stmaryrd}
\usepackage{xcolor}

\theoremstyle{definition}
\newtheorem{thm}{Theorem}[section]

\newtheorem{prop}[thm]{Proposition}
\newtheorem{cor}[thm]{Corollary}

\newtheorem{example}[thm]{Example}
\newtheorem{question}[thm]{Question}

\newcommand{\1}{\mathds{1}}
\newcommand{\OO}{\mathcal{O}}

\newtheorem{definition}[thm]{Definition}

\newtheorem{rem}[thm]{Remark}

\newcommand{\End}{\mathrm{End}}
\newcommand{\h}{\mathfrak{h}} 
\newcommand{\A}{\mathbf{A}}

\DeclareMathOperator{\E}{\mathlarger{ \mathbb{E}}}

\newcommand{\barchi}{\overline{\chi}}
\newcommand{\bareta}{\overline{\eta}}

\newcommand{\Ind}{\operatorname{Ind}}
\newcommand{\p}{\mathfrak{p}}

\newcommand{\Z}{\mathbf{Z}}
\newcommand{\Q}{\mathbf{Q}}
\newcommand{\R}{\mathbf{R}}

\newcommand{\C}{\mathbf{C}}
\newcommand{\F}{\mathbf{F}}

\newcommand{\K}{\mathbf{K}}
\newcommand{\Op}{\operatorname{Op}}

\newcommand{\GL}{\mathrm{GL}}
\newcommand{\PGL}{\mathrm{PGL}}
\newcommand{\Hom}{\mathrm{Hom}}

\newcommand{\SO}{\mathrm{SO}}

\newcommand{\g}{\mathfrak{g}}

\renewcommand{\O}{\mathcal{O}}

\newtheorem{lemma}{Lemma}

\usepackage{hyperref}
\hypersetup{
    colorlinks,
    citecolor=blue,
    filecolor=blue,
    linkcolor=blue,
    urlcolor=blue
}

\begin{document}

\author{Trajan Hammonds}
\title{Relative Character Asymptotics Beyond Stability for $\PGL_2 \times \GL_1$}
\begin{abstract}
The asymptotics of relative characters for real Lie groups were studied for representations $(\pi, \sigma)$ arising from Gan-Gross-Prasad pairs $(G,H)$ by Nelson and Venkatesh. They successfully compute the asymptotics of relative characters whenever the conductor of the associated Rankin-Selberg $L$-function $L(\pi \boxtimes \sigma^\vee)$ lies in a stable locus, i.e. away from conductor dropping. In this paper, we express asymptotics for relative characters in the \textit{non-archimedean} setting for $(\PGL_2, \GL_1)$. The key new innovation is that our method overcomes the stability hypothesis and allows for significant conductor dropping. 
\end{abstract}

\maketitle
\tableofcontents

\section{Introduction}
Let $G = \PGL_2(\F)$ and $H = \GL_1(\F)$ embedded diagonally in $G$, with $\F$ a non-archimedean local field like $\Q_p$. For a representation $\pi$ of $G$ and a character $\chi$ of $H$ our two main objects of study will be the integral operator $\Op(a)$ and the linear functional $\ell_\chi$ defined as follows: For a function $a \colon \mathrm{Lie}(G)^\ast \to \C$, we define
\begin{align*}
    \Op(a) := \int_{x \in \mathrm{Lie}(G)} a^\vee(x) \pi(\exp x)\,dx.
\end{align*}
(See Section \ref{sec:functions} for more details). Meanwhile, the linear functional $\ell_\chi : \pi \to \C$ is given by 
\begin{align*}
    \ell_\chi(v) := \int_{\F^\times} W_v\begin{pmatrix} h & \\ & 1 \end{pmatrix} \overline{\chi}(h)\,d^\times h, 
\end{align*}
where $W_v$ is the image of $v$ in its Whittaker model (See Section \ref{sec:functionals} for more details). In this setting, we define the \textit{relative character} as the distribution 
\begin{equation}\label{eqn:relchar}
 \mathcal{H}_{\pi, \chi}(a) = \sum_{v \in \mathcal{B}(\pi)} \ell_\chi(\Op(a)v)\overline{\ell_\chi(v)}.
\end{equation}

In Harish-Chandra's character theory, the asymptotics of the distributional character $\Theta_\pi$ of an irreducible representation $\pi$ of a $p$-adic group $G$ are controlled by the geometry of the nilpotent cone. In particular, let $\g$ be the Lie algebra of $G$ and $\F$ a $p$-adic field. Let $\mathcal{N}$ be the nilpotent cone of $\g$ and $\mathcal{N}/G$ the (finite) set of $G$-conjugacy classes. For $\OO \in \mathcal{N}/G$, denote $\mu_\OO$ the nilpotent orbital integral. Harish-Chandra showed that for $f$ supported sufficiently close (in terms of $\pi$) to the identity in $G$, 
\begin{equation}
    \Theta_\pi(f) := \mathrm{tr}(\pi(f)) = \sum_{\OO \in \mathcal{N}/G} c_\OO(\pi)\mu_\OO(a),
\end{equation}
where $a$ is such that $a^\vee$ is the Fourier transform of the pullback of $f$ to $\g$ under the exponential map. 
 One can think of the distributional character as given by the following composition
 \begin{equation}
 \Theta_\pi : C_c^\infty(G) \xrightarrow[]{f \mapsto \pi(f)} \End(\pi) \xrightarrow[]{T \mapsto \sum_{v} Tv \otimes v} \pi \otimes \overline{\pi} \xrightarrow[]{v_1 \otimes \overline{v_2} \mapsto \langle v_1, v_2 \rangle} \C.
 \end{equation}
 These characters are related to geometry via the Kirillov character formula. It says for an irreducible representation $\pi$ of $G$, one can attach a coadjoint orbit (or finite union) $\OO_\pi \subseteq \mathrm{Lie}(G)^\ast$  to $\pi$ such that 
 \begin{equation}
 \Theta_\pi(f) = \int_{\OO_\pi} (a^\vee j^{-1/2})^\vee
 \end{equation}
 where $j$ is the Jacobian of the exponential map $\exp : \mathrm{Lie}(G) \to G$. Kirillov's formula gives an explicit analogue of Harish-Chandra's character expansion by providing a uniform asymptotic at scales $x \ll 1/T$ where $T$ is the highest weight of the representation $\pi$. 
 
 For a symmetric pair $(G,H)$, the analogous distribution is the \textit{relative character} $\mathcal{H}_{\pi, \sigma}$ attached to a pair of representations $(\pi, \sigma)$ of $(G, H)$. One can define $\mathcal{H}_{\pi, \sigma}$ completely analogously to $\mathcal{H}_{\pi, \chi}$ above by replacing the functional $\ell_\chi \in \Hom_H(\pi, \C)$ with $\ell_\sigma \in \Hom_H(\pi, \sigma)$. However, we give an alternative definition below, after first stating a central question:

\begin{question}\label{question1}
  What are the asymptotics of relative characters? Is there an analogous expansion into geometric orbits?
\end{question}

 To help answer this question, we give an alternative definition of the relative character. Given a tempered irreducible representation $\pi$ of $G$, we can decompose it along a subgroup $H$, by 
 \begin{equation}
 \pi\mid_H = \int_\sigma m_\pi(\sigma) \sigma \,d\sigma,
 \end{equation}
 where $\sigma$ ranges over tempered irreducible representations of $H$. Then the inner product 
 \begin{equation}
 \langle v_1, v_2 \rangle = \int_\sigma m_\pi(\sigma) \mathcal{H}_{\pi,\sigma}(v_1 \otimes v_2)\,d\sigma,
 \end{equation}
 where $\mathcal{H}_{\pi,\sigma}(v_1 \otimes v_2)$ is the relative character given by 
 \begin{equation}
 \int_H \langle \pi(h)v_1, v_2 \rangle \overline{\chi_\sigma(h)}\,d\sigma. 
 \end{equation}
 From this, we obtain a distribution by the composition 
 \begin{equation}
 \mathcal{H}_{\pi, \sigma} : C_c^\infty(G) \to \End(\pi) \cong \pi \otimes \overline{\pi} \xrightarrow[]{\mathcal{H}_{\pi, \sigma}} \C.
 \end{equation}
 Thus, $\mathcal{H}_{\pi, \sigma}(f) = \int_H \rm{tr}(\pi(h)\pi(f))\overline{\chi_\sigma(h)}\,dh$. Now assuming for the moment that $\exp : \g \to G$ is an isomorphism with trivial Jacobian, we'd expect by the Kirillov formula, that
\begin{align}
  \mathcal{H}_{\pi, \sigma}(f) &= \int_{h \in H} \mathrm{tr}(\pi(h)\pi(f))\overline{\chi_\sigma}(h)\,dh = \int_{x \in \g} a^\vee(x) \int_{y \in \h} \chi_\pi(e^{y+x})\chi_\sigma(e^{-y})\,dx\,dy \\
                     &= \int_{x \in \g} a^\vee(x)\int_{y \in \h} \int_{\xi \in \OO_\pi} e^{i(x+y)\xi} \int_{\eta \in \OO_\sigma} e^{-iy\eta} = \int_{\OO_{\pi, \sigma}}a
\end{align}
where $\OO_{\pi, \sigma}$ is the set of matrices $\xi \in \OO_\pi$ whose restriction $\xi_H$ to $\h^\ast$ is $\eta \in \OO_\sigma$. So it seems that the \textit{relative} coadjoint orbits $\OO_{\pi, \sigma}$ may be good candidates for the geometric orbits we are looking for. We call $\OO_{\pi, \sigma}$ stable if it is an $H$-torsor. Question \ref{question1} is subtle even in small rank. We begin with an example illustrating the challenges that arise beyond the stable setting. 

\begin{example}\label{ex:nilcones}
    Let $(G,H) = (\GL_3(\R), \GL_2(\R))$ and pick representations $\pi$ and $\sigma$ so that the coadjoint orbits are the full nilpotent cones, i.e. $\OO_\pi = \mathcal{N} \subset \g^\ast$ and $\OO_\sigma = \mathcal{N}_H \subset \h^\ast$. The \textit{relative} coadjoint orbit is the set 
    \begin{equation}
        \OO_{\pi, \sigma} = \mathcal{N} \cap \rm{pr}^{-1}(\mathcal{N}_H)  = \left\{ \xi =
      \begin{pmatrix}
        A & b \\
        c & d
      \end{pmatrix}  \in \mathfrak{sl}_3(\mathbb{R}) :
      \xi,A
      \text{ are nilpotent}
      \right\} \subset \g^\ast, 
    \end{equation}
    where $\mathrm{pr} : \g^\ast \to \h^\ast$ is the natural projection map. 
    There are infinitely many $H$-orbits on $\OO_{\pi, \sigma}$ given by the representatives $\left\{\begin{pmatrix} 0 & 1 & b \\ 0 & 0 & 0 \\ 0 & c & 0 \end{pmatrix}\right\}$. 
  \end{example}

Therefore, due to the geometry of $\OO_{\pi, \sigma}$ in this example, there is no evident candidate for a single geometric orbit that might control the asymptotics of $\mathcal{H}_{\pi, \sigma}$. It's not exactly clear what the correct asymptotic picture should be for this pair, yet we remain hopeful that the techniques from this paper could elucidate the situation. 

For $p$-adic symmetric spaces, Rader and Rallis \cite{RR} developed an analogue of Harish-Chandra's character theory.  For $H$-distinguished representations $\pi$ of $G$, they show that on the $H$-regular semisimple locus, relative characters are locally constant and near the identity, they admit a germ expansion analogous to the Harish-Chandra case via
\begin{equation}
   H_{\pi}( \exp X )
   \;\sim\;
   \sum_{\mathcal{O}\subset\mathcal{N}(\mathfrak{s})}
   c_{\mathcal{O}}(\pi)\,\widehat{\mu}_{\mathcal{O}}(X),
\end{equation}
where $\mathfrak{s}$ is the $(-1)$-eigenspace of the involution defining the
symmetric pair, $\mathcal{N}(\mathfrak{s})$ is its nilpotent cone, and
$\widehat{\mu}_\mathcal{O}$ are Fourier transforms of $H$-invariant nilpotent orbital integrals. Despite this, a systematic theory that includes cases like the example above is still elusive, owing to the unstable geometry of the relative coadjoint orbit $\OO_{\pi, \sigma}$. However, in some (stable) cases, there is a general theory due to the work of Nelson and Venkatesh \cite{NV}. For real Gan--Gross--Prasad pairs such as $(\SO_{n+1},\SO_n)$ and $(\mathrm{U}_{n+1},\mathrm{U}_n)$,
Nelson and Venkatesh developed a theory of relative characters using the orbit method and microlocal analysis. For a stable pair of representations $(\pi_T, \sigma_T)$, indexed by their highest weight, Nelson and Venkatesh study the asymptotic behavior of $\mathcal{H}_{\pi, \sigma}(a_T)$ for a $T$-dependent function $a$. They prove for any $a$ supported in a fixed compact subset $U \subset \g^\ast$ of stable elements, as $T \to \infty$
\begin{align*}
    \mathcal{H}_{\pi, \sigma}(a_T) = \int_{\OO_{\pi, \sigma}} a + O\left(T^{2\delta - 1}\right)
\end{align*}
for some fixed $0 \leq \delta < 1/2$.  Outside the stable locus, the relative orbit becomes singular and the analysis degenerates. This motivates studying special cases as a step towards broader understanding to what extent one should expect the asymptotics of relative characters to be governed by a single geometric orbit even when stability fails. We can now state our main theorem which provides an instance of a relative character asymptotic  in a non-stable regime. 

\begin{thm}\label{thm:main theorem}
 Let $\pi$ be a principal series or supercuspidal unitary representation of $G$ and $\chi$ a unitary character of $H$. If $\pi$ is principal series $\pi = \chi_0 \boxplus \chi_0^{-1}$, assume that $\chi \neq \chi_0, \chi_0^{-1}$ or an unramified twist. Let  $a \ \colon \mathrm{Lie}(G)^\ast \to \C$ be a function constant at scales $|T| = q^N$ with $N \geq c(\chi)/2$ and supported on elements $T\tau$ with $y$ and $z$ coordinates $\tau_y$ and $\tau_z$ satisfying $|\tau_y|, |\tau_z| \leq \min\left(q^{c(\chi_0\chi^{-1})/2}, q^{c(\chi_0^{-1}\chi^{-1})/2}\right)$ in the principal series case and $|\tau_y|, |\tau_z| \leq (q^{c(\pi \otimes \overline{\chi})/2})$ in the supercuspidal case. Then, there exists a hyperbola (see \ref{sec:integral}) $\mathrm{Hyp}(\pi, \chi) \subset \mathrm{Lie}(G)^\ast$ so that
\begin{align}
	\mathcal{H}_{\pi, \chi}(a) = \int_{\mathrm{Hyp}(\pi, \chi)} a(\xi)\,d\xi.
\end{align}  
 \end{thm}
 When $F = \R$, the analogue of this theorem was proven by Nelson and Venkatesh in \cite{NV} for pairs of representations $(\pi, \sigma)$ of $(G,H)$ in the stable case. In this case, the hyperbola in Theorem \ref{thm:main theorem} has the following interpretation as a relative coadjoint orbit. Let $(G,H) = (\PGL_2(\R), \GL_1(\R))$ and $\pi$ is principal series and $\sigma$ is a character, and denote by $\mathbf{C}(\pi)$ and $\mathbf{C}(\sigma)$ the conductors of $\pi$ and $\sigma$. The coadjoint orbit for $\pi$ is given by 
\begin{equation}
    \OO_\pi := \left\{\begin{pmatrix} x & y+z \\ y-z & -x \end{pmatrix} : x^2 + y^2-z^2 = \mathbf{C}(\pi)^2\right\} \subseteq \g^\ast \cong\R^3,
\end{equation}
and the coadjoint orbit for $\sigma$ is given by
\begin{equation}
    \OO_\sigma := \{ x = \mathbf{C}(\sigma)\} \subseteq \h^\ast \cong \R.
\end{equation}
When $\mathbf{C}(\sigma) \neq \mathbf{C}(\pi)$, the relative coadjoint orbit is$$\OO_{\pi, \sigma} := \OO_\pi \cap \rm{pr}^{-1}(\OO_\sigma) =  \{y^2 - z^2 = \mathbf{C}(\pi)^2 - \mathbf{C}(\sigma)^2\},$$
which carves out a hyperbola, which is the analogue of our $\mathrm{Hyp}(\pi, \chi)$ (see \ref{fig:onesheet}).
\begin{figure}
    \centering
    \includegraphics[scale=.2]{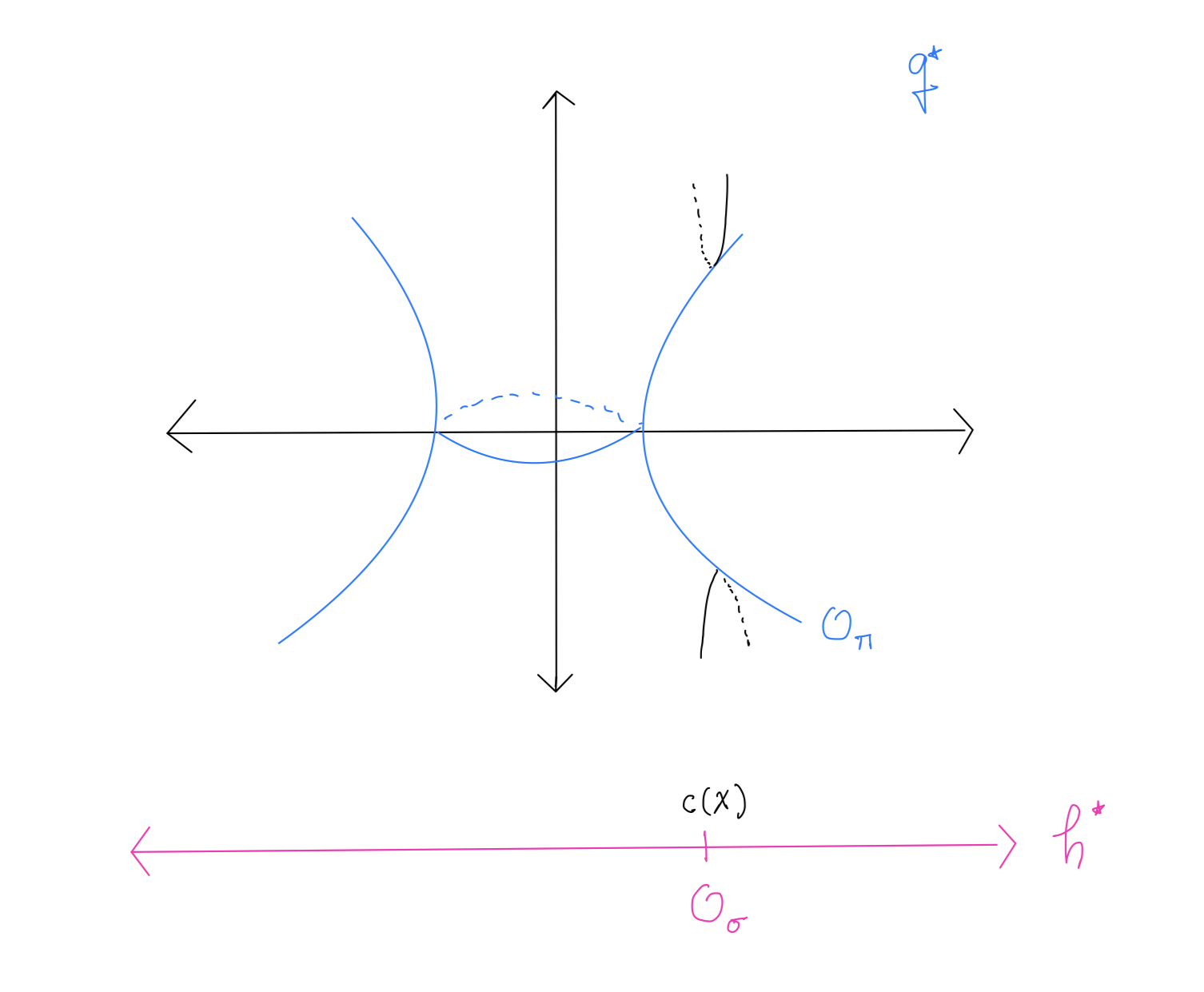}
    \caption{Relative Coadjoint Orbit}
    \label{fig:onesheet}
\end{figure}
However, when $\mathbf{C}(\sigma) = \mathbf{C}(\pi)$ the resulting locus is a cross which has a singularity at the origin. In \cite{NV}, they have to rule out this degeneracy by working with \textit{stable} pairs of representations $(\pi,\sigma)$, that is, those representations for which $H$ acts simply transitively on $\OO_{\pi, \sigma}$.  This notion of stability amounts to a condition on the infinitesimal characters $\lambda_\pi$ and $\lambda_\sigma$ of $\pi$ and $\sigma$: namely that $\{ \text{ eigenvalues of } \lambda_\pi\} \cap \{\text{ eigenvalues of } \lambda_\sigma\} = \emptyset$. Moreover, we say a pair $(\lambda, \mu) \in \g^\ast \times \h^\ast$ is stable if $\{\text{ eigenvalues of } \lambda \} \cap \{\text{ eigenvalues of } \mu\} = \emptyset.$ 

\begin{rem}
    Away from stable pairs is where the conductor $C(\pi \times \overline{\sigma})$ of the Rankin-Selberg $L$-function $L(s, \pi \times \overline{\sigma})$ drops. Recent subconvexity results such as \cite{Unpaper} rely on these relative character asymptotics away from the stable range and thus require the \textit{non-conductor dropping} assumption 
\begin{align*}
    \forall i, j \ |\lambda_{\pi, i} - \lambda_{\sigma,j}| \gg T.
\end{align*}
The advantage of our main theorem is that it provides a relative character asymptotic in the case where the conductor is allowed to drop significantly. Indeed, the assumption on $\pi$ and $\chi$ in our theorem is analogous to the non-conductor dropping requirement 
\begin{align*}
    \forall i, j \ |\lambda_{\pi, i} - \lambda_{\sigma, j}| \gg 1,
\end{align*}
but only in the non-archimedean case for $(G,H) = (\PGL_2, \GL_1)$. In this case $\lambda_{\pi, i}$ is just the conductor exponent of $\chi_0$ or $\chi_0^{-1}$ and $\lambda_{\sigma, j}$ is the conductor exponent of $\chi$. 
\end{rem}

There is a difference in the failure of stability in Theorem \ref{thm:main theorem} vs in Example \ref{ex:nilcones}. Namely in Theorem \ref{thm:main theorem}, the cross shape fails to be an $H$-torsor because at $0$ it does not have trivial stabilizer and the orbit not containing zero is not closed, while in Example \ref{ex:nilcones}, stability fails because there are infinitely many $H$-orbits on $\OO_{\pi, \sigma}$ hence the stabilizers aren't even necessarily finite! 
\begin{question}
    To what extent does Theorem \ref{thm:main theorem} generalize to other non-stable situations?
\end{question}

Even though stability fails in Theorem \ref{thm:main theorem}, there is still a natural orbit. We may informally conjecture, then, that relative character asymptotics should be governed by the non-closed orbit with finite stabilizers, which covers situations like Theorem \ref{thm:main theorem} but not Example \ref{ex:nilcones}. 

The strategy of Nelson and Venkatesh has most notably led to many breakthroughs on the subconvexity problem including \cite{Unpaper}, \cite{standard}, \cite{Marshall}, \cite{HN} to name a few. We don't pursue any applications to subconvexity in this paper and entirely avoid any relative trace formula techniques. For an in-depth introduction to the subject we recommend the introduction of the original paper \cite{NV} of Nelson and Venkatesh, as well as the thorough exposition in the recent preprint \cite{AT}, which also considers non-archimedean microlocal analysis. 

\subsection{Sketch of the Proof}\label{subsec:proofsketch}
By formally applying Parseval to Equation \ref{eqn:relchar}, it suffices to compute 
$$\langle \Op(a) \ell_\chi, \ell_\chi \rangle,$$
for $a$ as in Theorem \ref{thm:main theorem}. Such functions have a decomposition into wave packets (Proposition \ref{prop:wavepacketdecomp}), so by linearity we can compute the relative character for a single wave packet $1_\tau^T$. The function $1_\tau^T$ is a bump function around some fixed element $\tau \in \g^\ast$. In fact, $1_\tau^T$ is the Fourier transform of a wave packet supported on the congruence subgroup $K(N)$ oscillating at fixed frequency $\tau$ (See Section \ref{sec:functions}). To compute the relative character for $a = 1_\tau^T$, we will factor $\Op(a)$ into directions, corresponding to the Iwahori factorization $K(N) = K_{-}K_0K_+$. For any function $b : F \to \C$ and for $\star \in \{-, 0, +\}$, $\Op^\star(b)$ averages the action of $K_\star$, weighted by $b$ (see Equation \ref{eqn:Opstarb} for a precise definition). Therefore, we factor $a = a_- a_0 a_+$ into such functions, each of which is a one-variable function on $\mathrm{Lie}(G_\star)^\ast$. The function $a_+$ (resp. $a_-$) most naturally pairs with the operator $\Op^-$ (resp. $\Op^+$) due to the duality between $\g$ and $\g^\ast$ (see Section \ref{subsec:opcalculus}). We will use the Weyl element $w = \begin{pmatrix} 0 & 1 \\ 1 & 0 \end{pmatrix}$ and the $\GL_2 \times \GL_1$ local functional equation to go between $\Op^-$ and $\Op^+$ (see Section \ref{subsec:weyl}). In particular, we have the relation $\Op^+ = \pi(w)\Op^-\pi(w)$. These decompositions and factorizations allow us to express the relative character as 
\begin{align*}
\langle \Op(a)\ell_\chi, \ell_\chi\rangle &= \langle \Op^-(a_+)\Op^0(a_0)\Op^+(a_-)\ell_\chi, \ell_\chi \rangle \\
&= \langle \pi(w)\Op^0(a_0)\Op^+(a_-)\ell_\chi, \Op^+(a_+)\pi(w)\ell_\chi \rangle \\
&= \gamma(\pi \otimes \overline{\chi})\langle \pi(w)\Op^+(a_-)\Op^0(a_0)\ell_\chi, \Op^+(a_+)\ell_{\chi^{-1}} \rangle \\
&= \gamma(\pi \otimes \overline{\chi}) a_0(\tau_x)\langle \pi(w)\Op^+(a_-)\ell_\chi, \Op^+(a_+)\ell_{\chi^{-1}} \rangle. 
\end{align*}
Denote this last inner product as 
\begin{align*}
I = \langle \pi(w)\Op^+(a_+)\ell_\chi, \Op^+(a_-)\ell_{\chi^{-1}}\rangle .	
\end{align*}
From here, we work in the Kirillov model $\K(\pi, \psi)$ for $\pi$, since this model diagonalizes the action of the unipotent subgroup $G_+$ in $G$. Indeed, after identifying $\ell_\chi$ with a suitable vector in the Kirillov model, we essentially have 
\begin{align*}
    \Op^+(a_+)\ell_\chi &=  \{ h \mapsto a_+(h)\chi(h)\} \in \K(\pi, \psi)
\end{align*}
and
\begin{align*}
    \Op^+(a_-) \ell_{\overline{\chi}}= \{h \mapsto a_-(h)\overline{\chi}(h)\} \in \K(\pi, \psi).
\end{align*}
Then we can expand into a basis of characters and applying Mellin inversion leads us to study
\begin{align*}
    I = \int_{\eta} \langle a_-\barchi, \eta \rangle \langle \eta, \pi(w)a_+\chi \rangle \,d\eta  = \int_{\eta} \gamma(\pi \otimes \overline{\eta}) \langle a_-\barchi, \eta \rangle \langle \bareta, a_+ \barchi \rangle \,d\eta.
\end{align*}
At this point we crucially use that we are working with a specific choice of test function $a$, that is $a = 1_\tau^T$. The technical bulk of this paper is spent calculating $I$, which we carry out in Section \ref{sec:main calculations}. 

On the other hand, evaluating the integral on the right hand side of Theorem \ref{thm:main theorem} is relatively straightforward. We conclude the proof by carrying out the computation of the right-hand side and comparing it with the previous computation of the left-hand side. 

\subsection{Organization of the Paper}
 In Section \ref{sec:preliminaries} we  give some important definitions and set-up for the rest of the paper.  In Section \ref{sec:functions} , we  explicitly describe the functions we work with, and examine some of their properties, along with proving the desired properties of the operators $\Op(a)$. In Section \ref{sec:functionals}, we prove some of the desired properties of the linear functionals and show how they are related to the relative character. In Section \ref{sec:main calculations}, we  carry out the calculations necessary to complete the proof sketch from the previous section. Finally, in Section \ref{sec:integral}, we compute the corresponding integral in phase space and show that it agrees with the relative character. 

\section*{Acknowledgements}
We thank Paul Nelson and Akshay Venkatesh for helpful conversations and comments on earlier drafts. This work constitutes the author's PhD thesis and was revised while supported by a research grant (VIL54509) from VILLUM FONDEN. 

\section{Preliminaries}\label{sec:preliminaries}
Let $v$ be a finite place of a global field $\F$ and let $\F_v$ denote the corresponding non-archimedean local field. We only work locally so, by abuse of notation, we'll refer to $\F_v$ as just $F$. Let $\OO$ be its ring of integers with maximal ideal $\mathfrak{p}$. We denote by $\varpi$ its uniformizer. The absolute value on $F$ will be denoted $|\cdot|$ (with the subscript sometimes omitted) so that $|\varpi| = q^{-1}$, where $q$ is the size of the residue field $\OO/\mathfrak{p}$. Let $\mathrm{val}(x) = v(x) : F \to \Z$ be such that $|x| = q^{-v(x)}$. For the purposes of this thesis, we will assume that $q$ is odd. \\
We will often use $\psi$ to denote an additive character of $F$, i.e. $\psi(x + y) = \psi(x)\psi(y)$. We let $\mathrm{cond}(\psi)$ denote the largest ideal of $\OO_F$ on which $\psi$ is trivial and denote its conductor exponent $\rm{c}(\psi)$ so that $\rm{cond}(\psi) = \mathfrak{p}^{\rm{c}(\psi)}$. If $\rm{cond}(\psi) = \OO_F$, we say $\psi$ is unramified, and we have $\rm{c}(\psi) = 0$. Letters like $\chi, \nu,$ or $\eta$ will denote multiplicative characters of $F^\times$. We denote the multiplicative subgroup $\OO^\times \cap (1 + \p^n)$ by $U(n)$. We let $\rm{cond}(\chi)$ denote the largest subgroup of $\OO_F^\times$ on which $\chi$ is trivial and denote its conductor exponent $\rm{c}(\chi)$ so that $\rm{cond}(\chi) = U(c(\chi))$. If $\rm{cond}(\chi) = \OO_F^\times$, then $\rm{c}(\chi) = 0$ and we say that $\chi$ is unramified. Every unramified character $\eta$ can be written as $\eta(x) = |x|^s$ for some $s \in \C$. Let $X(F^\times)$ denote the group of multiplicative characters of $F$. The unramified characters of $F^\times$ form a subgroup of $X(F^\times)$ isomorphic to $\C/(\Z \cdot 2\pi i /\log q)$. \\
There is a unique Haar measure $\,du$ on the additive group $F$ normalized so that $\mathrm{vol}(\OO, \,du) = 1$. The multiplicative Haar measure $\,d^\times u$ on $F^\times$ is normalized so that $\mathrm{vol}(\OO^\times, \,d^\times u) = 1$. The measures satisfy $\,d^\times u = \zeta_F(1)\frac{\,du}{|\cdot|}$, where $\zeta_F(1) = (1-q^{-1})^{-1}$ is the local zeta factor. With these measures, we have $\mathrm{vol}(\p^n, \,du) = q^{-n}$ and $\mathrm{vol}(U(n), \,d^\times u) = \zeta_F(1)q^{-n}$. Let $\mathcal{X}_n$ denote the set of characters with conductor at most $n$. Then $|\mathcal{X}_n| = \zeta_F(1)^{-1}q^n$.  \\
We define the congruence subgroups 
\begin{align*}
    K^0[m_v] := \left\{ \begin{pmatrix}
        a & b \\ c & d 
    \end{pmatrix} \in G(\OO_v) \mid c \equiv 0 \bmod \varpi_v^{m_v}\right\}.
\end{align*}
If $\pi_v^{K^0[m-1]}$ is empty but $\pi_v^{K^0[m]}$ is nonzero, then we call $m$ the conductor exponent of $\pi$  and write $C(\pi) = \mathfrak{p}^{v(\pi)}$. The main subgroup of interest for us is the principal congruence subgroup, defined as 
\begin{align*}
    K(n) := \left\{ \begin{pmatrix} a & b \\ c & d \end{pmatrix} \in G(\OO_v) \mid a-1, b, c, d-1 \equiv 0 \bmod \varpi_v^n \right\}.
\end{align*}
We have a decomposition of $K_v(n)$ into its Iwahori factorization given by 
$$K(n) = N(\varpi^n)A(1+\varpi^n)N_{-}(\varpi^n),$$
where for $\mathfrak{a} \subset \OO$, $N(\mathfrak{a})$ has upper right entries in $\mathfrak{a}$, $A(\mathfrak{a})$ has diagonal entries in $1 + \mathfrak{a}$ and $N_{-}(\mathfrak{a})$ has bottom left entries in $\mathfrak{a}$. \\

We denote 
$$N := \left\{ n(x) = \begin{pmatrix} 1 & x \\ 0 & 1 \end{pmatrix} : x \in F\right\} \ \text{ and } A := \left\{ a(y) = \begin{pmatrix} y & 0 \\ 0 & 1 \end{pmatrix} y \in F^\times\right\}.$$

\subsection{Whittaker and Kirillov Models}
For a generic irreducible representation $\pi$ and additive character $\psi$ of $F$, we embed $\pi$ into its smooth Whittaker model $\mathcal{W}(\pi, \psi)$ given by the space of functions (see \cite{JS83} or \cite[Sec 37]{BH})
$$\mathcal{W}(\pi, \psi) = \{ W \in C^\infty(G) : n(x)W(g) = \psi(x)W(g) \text{ for all } (x,g) \in F \times G\}.$$
The restriction map $\mathcal{W}(\pi, \psi) \to \{\text{ functions } A_F \to \C\}$ is injective and its image \textit{contains} $C_c^\infty(A_F)$ but does not consist solely of such functions. Therefore each $W \in \mathcal{W}(\pi, \psi)$ is determined by a function $W : F^\times \to \C$ so that
$$W(y) := W(a(y)).$$
Moreover, every compactly supported smooth function on $A_F$ arises this way. This realization of $\pi$ as a space of functions on $F^\times$ is the Kirillov model of $\pi$. For unitary $\pi$, an invariant inner product is given by
$$\langle W_1, W_2 \rangle = \int_{F^\times} W_1(y)\overline{W_2(y)}\,d^\times y.$$ 

\subsection{$\GL_1$ Iwasawa-Tate Theory}
If $\chi$ is a multiplicative character, then its twisted dual is $\chi^\vee := \chi^{-1}|\cdot|$. Let $\,d^\times x$ be the Haar measure on $F^\times$, which we take to be $\frac{\,dx}{|x|}$ where $\,dx$ is the Haar measure on $F$. Define local zeta functionals $Z(f, \chi)$ for $f$ a Schwartz function on $F$ by 
\begin{align*}
    Z(f, \chi) = \int_{F^\times} f(x)\chi(x)\,d^\times x. 
\end{align*}
We next recall the local $L$-factors; $L(\chi) = (1 - \chi(\varpi))^{-1}$ if $\chi$ is unramified and $L(\chi) = 1$ if $\chi$ is ramified. Then the $\GL_1$ local functional equation states that 
\begin{align*}
    \frac{Z(f,\chi)}{L(\chi)}\epsilon(\chi, \psi) = \frac{Z(\hat{f}, \chi^\vee)}{L(\chi^\vee)},
\end{align*}
where $\hat{f}$ denotes the Fourier transform $\hat{f}(y) = \int_F f(x)\psi(xy)\,dx.$ It is important to understand the epsilon factor which will turn out to play a privileged role in our analysis. To do this we first define the Gauss sum (see \cite{GH})
\begin{align*}
    g(\nu, \psi) = \int_{\OO^\times} \nu(x)\psi(x)\,d^\times x.
\end{align*}
It has the property that if the conductors of $\nu$ and $\psi$ are different, then the Gauss sum vanishes. If the conductors are say, both $\p^n$, then $|g(\nu, \psi)|^2 = q^{-n}$. When $\chi$ is replaced by $\chi|\cdot|^{s-1/2}$, we write $\epsilon(s,\chi, \psi)$ for the epsilon factor. 
\begin{lemma}
    Suppose $\chi$ is a multiplicative character of the form $\eta |\cdot|^s$ with conductor $\mathrm{c}(\chi) = n$. Let $\psi$ be an unramified additive character. Then 
    \begin{equation}
        \epsilon(s,\chi, \psi) = q^{n(1-s)}\eta(\varpi^n)g(\eta^{-1}, \psi_{\varpi^{-n}}),
    \end{equation}
    where $\psi_t(x) := \psi(tx)$. 
\end{lemma}
\begin{proof}
See \cite{GH} Thm 2.3.8.
\end{proof}
We will use the following corollary often, whose proof is immediate. 
\begin{cor}
    For $\chi$ unitary, the epsilon factor satisfies 
    \begin{align*}
        |\epsilon(1/2, \chi, \psi)|^2 = 1.
    \end{align*}
\end{cor}

We will also use the fact that $\epsilon(s, \chi |\cdot|^a) = q^{-ac(\chi)}\epsilon(s, \chi)$. Lastly, we need two propositions to understand the interaction between multiplicative characters and additive characters. 

\begin{prop}\cite[Sec 1.8]{BH}\label{prop:chi}
    Let $\chi$ be a multiplicative character with conductor exponent $c(\chi) = n$. Then for a fixed unramified additive character $\psi$, there is $\alpha_\chi \in F$ such that 
    \begin{align*}
        \chi(1+x) = \psi(\alpha_\chi x),
    \end{align*}
    for $x \in \p^{n/2 +1}$. 
\end{prop}

We remark that such an $\alpha_\chi$ is not unique and is only defined up to $\p^{-n/2-1}$. 

\begin{rem}
    We briefly remark on why something like Proposition \ref{prop:chi} might be useful. Consider in the real case, the integral 
    \begin{equation}
        G(\lambda, t) := \int_0^\infty |x|^{i\lambda}e^{-itx}\,dx. 
    \end{equation}
    In order to do stationary phase, we write the multiplicative character $|x|^{i \lambda}$ as the additive character $e^{i\lambda \log(x)}$, giving us a phase function $\Phi(x) = i(\lambda \log(x) -tx)$ and a stationary point $x = \lambda/t$. Analogously, by writing a multiplicative character $\chi$ as an additive one, we can conduct non-archimedean stationary phase which becomes exact. 
\end{rem}

\begin{prop}\cite[Sec 3]{Ta}\label{prop:tate}
     Let $\chi$ and $\omega$ be two characters with $m = c(\omega) \leq c(\chi)/2 = n/2$. Assume $\psi$ is unramified and $n$ is even. Let $\alpha_\chi$ be such that $\chi(\exp a) = \psi(\alpha_\chi a)$ for $a \in \p^{n/2}/\p^n$. Then 
     $$\epsilon(1/2, \chi\omega, \psi) = \omega(\alpha_\chi)\epsilon(1/2, \chi, \psi).$$
\end{prop}
\begin{rem}
    We can again interpret Proposition \ref{prop:tate} as a stationary phase result analogous to what happens in the real case. Indeed, suppose we twist $G(\lambda, t) = \int_0^\infty |x|^{i\lambda}e^{-itx}$ by a small character $|x|^{i\nu}$ with $\nu \ll \lambda^{1/2}$. By the same reasoning, this twisted integral 
    \begin{align*}
        G(\lambda, \nu, t) &= \int_0^\infty |x|^{i\lambda}|x|^{i\nu} e^{-itx}\,dx 
    \end{align*}
    has a stationary point at $x = \frac{\lambda + \nu}{t}$. The contribution from this stationary point is 
    \begin{equation*}
        |\lambda + \nu|^{-1/2} \left|\frac{\lambda + \nu}{t}\right|^{i(\lambda + \nu)}e^{-i(\lambda + \nu)}. 
    \end{equation*}
    We can write this as 
    \begin{align*}
        |\lambda|^{-1/2}\left|(1 + \frac{\nu}{\lambda})\right|^{-1/2} |\lambda|^{i(\lambda + \nu)}\left|(1 +\frac{\nu}{\lambda})\right|^{i(\lambda + \nu)}e^{-i\lambda}e^{-i\nu}
    \end{align*}
    and after applying some Taylor expansions and using that $\nu \ll \lambda^{1/2}$, this is approximately the same as 
    \begin{align*}
        |\lambda|^{-1/2} \left|\frac{\lambda}{t}\right|^{i\nu} \left|\frac{\lambda}{t}\right|^{i\lambda}e^{-i\lambda} \approx \left|\frac{\lambda}{t}\right|^{i\nu}G(\lambda, t).
    \end{align*}
    So indeed, we have $G(\lambda, \nu, t) \approx \left|\frac{\lambda}{t}\right|^{i\nu}G(\lambda, t)$ provided $\nu \ll \lambda^{1/2}$. 
\end{rem}

\subsection{Local Representations}
In order to carry out the calculations of the relative character we need to associate some data to our representations. Namely, we need to understand the conductor $\mathrm{c}(\pi \otimes \chi)$, and the elements $\alpha_{\pi, \chi} \in F$ which will appear in the definition of $\mathrm{Hyp}(\pi, \chi)$. We have already seen the definition of $\alpha_\chi$ which appeared at the end of the previous subsection. 

 Let $\pi$ be a generic supercuspidal representation of $\GL_2(F)$ with trivial central character. It corresponds to a two-dimensional representation $\rho = \Ind_{E/F} \xi := \Ind_{\mathcal{W}_E}^{\mathcal{W}_F} \xi$ under the Local Langlands correspondence, where $\mathbf{E}$ is a quadratic extension of $F$ and $\xi$ a character of $\mathbf{E}^\times$. Note that since we work with $\PGL_2$, we require $\det \rho = 1$. We can view $\xi$ as a character of the Weil group $\mathcal{W}_E$ by class field theory. Explicitly, if $\mathbf{\alpha}_\mathbf{E}$ is the Artin reciprocity map
 \begin{align*}
 \mathbf{\alpha}_\mathbf{E} : \mathcal{W}_\mathbf{E}^{\rm{ab}} \to \mathbf{E}^\times,
 \end{align*} 
 then $\xi \circ \mathbf{\alpha}_\mathbf{E}$ is a character of $\mathcal{W}_\mathbf{E}$. For any character $\eta$ of $F^\times$ or $\psi$ of $F$, let $\eta_\mathbf{E} := \eta \circ \mathrm{Nm}_{\mathbf{E}/F}$ and $\psi_\mathbf{E} := \psi \circ \mathrm{Tr}_{\mathbf{E}/F}$. Lastly let $e := e(\mathbf{E}|F)$ be the ramification index and $f := f(\mathbf{E}|F)$ be the inertial degree. The LLC works well with gamma and epsilon factors in the following sense (see \cite[Sec 34]{BH}): for any character $\nu$ on $F^\times$, 
 \begin{align*}
 \gamma(\pi \otimes \nu) = \gamma(\rho \otimes \nu) = 	\gamma(\xi\nu_E)
 \end{align*}
 and 
 \begin{align*}
 \epsilon(1/2, \pi \otimes \nu, \psi) = \epsilon(1/2, \xi \nu_E, \psi_E).	
 \end{align*}
In particular, the conductor satisfies $\mathrm{c}(\pi \otimes \barchi) = f\mathrm{c}(\xi\chi_E^{-1})$. 
\begin{definition}
We define $\alpha_{\pi, \chi}$ for $\pi$ supercuspidal as $\mathrm{Nm}(\alpha_{\xi\chi_E^{-1}})$ where $\alpha_{\xi\chi_E^{_1}}$ is such that 
\begin{align*}
\xi\chi_E^{-1}(1+x) = \psi(\alpha_{\xi\chi_E^{-1}}x),	
\end{align*}
for $x \in \p_\mathbf{E}^{c(\xi\chi_E^{-1})/2}$. 
\end{definition}    

The other case is that $\pi$ is principal series. In this case we write $\pi = \chi_0 \boxplus \chi_0^{-1} = \Ind_{B(F)}^{G(F)}(\chi_0, \chi_0^{-1})$. The gamma factor and epsilon factor both factor into the product of $\GL_1$ gamma and epsilon factors. Namely, 
\begin{align*}
	\gamma(\pi \otimes \nu) = \gamma(\chi_0\nu)\gamma(\chi_0^{-1}\nu)
\end{align*}
and
\begin{align*}
\epsilon(1/2, \pi \otimes \nu, \psi) = \epsilon(1/2, \chi_0 \nu, \psi) \epsilon(1/2, \chi_0^{-1}\nu, \psi). 	
\end{align*}
To understand $\mathrm{c}(\pi \otimes \barchi)$, let $\chi_\flat = \chi_0\chi^{-1}$ and $\chi_\sharp = \chi_0^{-1}\chi^{-1}$. Then $\mathbf{C}(\pi \otimes \barchi)$ is the sum of the conductors of $\chi_\flat$ and $\chi_\sharp$, i.e.
\begin{align*}
\mathbf{C}(\pi \otimes \barchi) = c(\chi_\flat) + c(\chi_\sharp).
\end{align*}
\begin{definition}
For $\pi$ principal series we define $\alpha_{\pi, \chi} = \alpha_{\chi_\sharp}\alpha_{\chi_\flat} \in F$, where $\alpha_{\chi_\flat}$ and $\alpha_{\chi_\sharp}$ are defined analogously as Proposition \ref{prop:chi}. 
\end{definition}

\subsection{$\GL_2$ Hecke-Jacquet-Langlands Theory}
First we address the global theory. Let $\F$ be the global field and $F = \F_v$ be the corresponding local field. Let $\varphi$ be a cusp form on $GL_2$, i.e. $\varphi \in L^2_{\mathrm{cusp}}(\GL_2(\F) \backslash \GL_2(\A))$. For $\chi$ a character of $\F^\times \backslash \A^\times$ we define the global zeta integral 
\begin{align*}
    Z(s, \varphi, \chi) = \int_{\F^\times \backslash \A^\times} \varphi(a(y))\chi(y)|y|^{s - 1/2}\,d^\times y.
\end{align*}
By the left-invariance of $w$, we have
\begin{align*}
   \varphi(a(y)) = \varphi(w a(y)) = \varphi(w a(y) w^{-1} w) = \omega(y)\varphi(a(y^{-1})w),
\end{align*}
where $\omega$ denotes the central character of $\pi$. Note that this relationship also forces $\varphi$ to decay at $0$, as well as at infinity (since it is a cusp form). 
Hence the global functional equation says that 
\begin{align*}
    Z(s, \varphi, \chi) = Z(1-s, \pi(w)\varphi, \chi^{-1}\omega^{-1}).
\end{align*}
From the global functional equation, we can deduce the local functional equation by using the factorization $\varphi$ in its Whittaker model. We have
\begin{align*}
    Z(s, \varphi, \chi) = \prod_v Z(s, W_{\varphi, v}, \chi_v),
\end{align*}
where for $\mathrm{Re}(\chi_v) + \mathrm{Re}(s) > \frac{1}{2}$
\begin{align*}
    Z(s, W_v, \chi_v) = \int_{\F_v^\times} W_v\begin{pmatrix} a &  \\ & 1 \end{pmatrix} \chi_v(a) |a|^{s-1/2}\,d^\times a. 
\end{align*}
When $v$ is unramified, 
\begin{align*}
    Z(s, W_v, \chi_v) = L(s, \pi_v \otimes \chi_v) = (1 - \mu_v\chi_v(\varpi)q_v^{-s})^{-1}(1 - \nu_v\chi_v(\varpi)q_v^{-s})^{-1},
\end{align*}
where $\pi_v = \rm{Ind}_{B(F_v)}^{G(F_v)}(\mu_v, \nu_v)$. The identity 
\begin{align*}
    Z(s, \varphi, \chi) = L(s, \pi \otimes \chi) \prod_{v} \frac{Z(s, W_v, \chi_v)}{L(s, \pi_v \otimes \chi_v)}
\end{align*}
leads to 
\begin{prop}{Local Functional Equation (see \cite[Thm 2.18]{JL})} \label{prop:funeq}
We have
\begin{align*}
   \frac{ Z(s, W_v, \chi_v)}{L(s, \pi_v \otimes \chi_v)} \epsilon(s, \pi_v , \chi_v, \psi_v) = \frac{Z(1-s, \pi(w)W_v, \omega_v^{-1}\chi_v^{-1})}{L(1-s, \pi_v \otimes \omega_v^{-1}\chi_v^{-1})},
\end{align*}
where $\epsilon(s, \pi_v, \chi_v, \psi_v)$ is of the form $Ae^{Bs}$. More succinctly, we may write
\begin{align*}
    Z(s, W_v, \chi_v) \gamma(s, \pi_v \otimes \chi_v, \psi_v) = Z(1-s, \pi(w)W_v, \omega_v^{-1}\chi_v^{-1}),
\end{align*}
where $\gamma(s, \pi_v \otimes \chi_v, \psi_v) = \epsilon(s, \pi_v, \chi_v, \psi_v)\frac{L(1-s, \pi_v \otimes \omega_v^{-1}\chi_v^{-1})}{L(s, \pi_v \otimes \chi_v)}$.
\end{prop}
Throughout, we will specialize the local functional equation to $s = 1/2$, reserving $s$ for a variable name later. The local functional equation provides a way to evaluate $\pi(w)W(a(y))$ via Mellin inversion. Explicitly, (see \cite{MV}), we have an equality of the form 
\begin{align*}
    \pi(w)W\begin{pmatrix} y & \\ & 1 \end{pmatrix} = \int_{\widehat{F^\times}} \gamma(1/2, \pi \otimes \chi) \omega\chi(y)\int_{F^\times} W\begin{pmatrix} a & \\ & 1 \end{pmatrix} \chi(a)\,d^\times a \,d\chi,
\end{align*}
where $\omega$ is the central character of $\pi$, which for us is trivial.
 \section{Test Functions and $\Op$ Calculus}\label{sec:functions}

In this section we define $\Op(a)$ for Schwartz functions $a$ on $\mathrm{Lie}(G)^\ast$. For our class of functions $a$, we write a decomposition of $a$ as $a = \sum a_\tau$ with each $a_\tau$ a bump function around $T\tau$. We then show for these functions we can write 
\begin{equation*}
    \Op(a_\tau) = \Op^-(a_{\tau,+})\Op^0(a_{\tau,0})\Op^+(a_{\tau, -})
\end{equation*} in any order, where each $a_{\tau, \star}$ is a one-variable function on $F$. 
 \subsection{Test Functions}
 Let $\g$ be the Lie algebra of $G$ with coordinates 
\begin{align}
\begin{pmatrix} p_x & p_y \\ p_z & -p_x \end{pmatrix} = p \in \g, \label{eqn:liecoord}
\end{align}
and the dual Lie algebra $\g^\ast$ with coordinates 
\begin{align}
	\begin{pmatrix} \xi_x & \xi_z \\ \xi_y & -\xi_x \end{pmatrix} = \xi \in \g^\ast. \label{eqn:dualliecoord}
\end{align}
For the Lie algebra, we will use the basis $e_0 = \begin{pmatrix} 1/2 & \\ & -1/2\end{pmatrix}$, $e_+ = \begin{pmatrix} & 1 \\ & \ \end{pmatrix}$ and $e_- = \begin{pmatrix} & \ \\ 1 & \end{pmatrix}$. 
 For a fixed unramified unitary character $\psi$ of $\mathcal{O}_F$, there is a natural bilinear pairing $\g \times \g^\ast \to 
\C^\ast$ given by the trace pairing 
\begin{align*}
\langle x, \xi \rangle = \psi(\mathrm{Tr}(x\xi)).	
\end{align*}
Note that the coordinates $y$ and $z$ in Equations \ref{eqn:liecoord} and \ref{eqn:dualliecoord} get switched when going from the Lie algebra to the dual Lie algebra or vice versa; this is due to the duality induced by the trace pairing defined above.
In the archimedean case, one would construct operators that project onto \textit{microlocalized} vectors, i.e. those satisfying
\begin{align}\label{eqn:microlocalisation}
\pi(\exp x) v \approx e^{ix\tau/h}v,	
\end{align}
for all $x \in \g$ of size $|x| = O(h)$. When Equation \ref{eqn:microlocalisation} holds, we say $v$ is microlocalized at $\tau$. We think of $h$ as an infinitesimal scaling parameter that goes to zero, so that we are picking up vectors with symmetry properties under a shrinking neighborhood of the identity. We will attempt to mimic this construction in the non-archimedean case.

\begin{rem}
In other recent work involving microlocal analysis in non-archimedean setting, authors have worked with \textit{minimal vectors} in the case of supercuspidal representations (c.f. \cite{HN}), and \textit{microlocal lift vectors} in the case of principal series representations (c.f. \cite{NelsonQUE}). Here we opt to work with vectors that exactly match the notion of microlocalisation in Equation \ref{eqn:microlocalisation}. 
\end{rem}

Let $k(N)$ be the set of matrices in $\g  \ \cap \p^N\mathrm{Mat}_2(\OO)$. Let $K(N)$ be the image of $k(N)$ under the exponential map. 

\begin{definition}\label{def:nonarchlocalized}
    A non-archimedean microlocalized vector is a vector satisfying 
    \begin{align*}
        \pi(\exp x)v = \langle -x, \tau \rangle v
    \end{align*}
    for all $x \in k(N)$ and some $\tau \in \g^\star$.
\end{definition}

\begin{rem}
    In comparison to the archimedean notion in Equation \ref{eqn:microlocalisation}, we should think of $h$ as $h=1/T$ where $T = \varpi^{-N}$. Then as $N \to \infty$, $|T| \to \infty$, and microlocalization becomes an eigenvector property of an increasingly small neighborhood of the origin. 
\end{rem}
We want to find an $a$ so that $\Op(a)$ projects onto non-archimedean microlocalized vectors. To that end, we pick $a = a_\tau$ so that its inverse Fourier transform is given by  
\begin{align*}
a^\vee_\tau(x) = \mathrm{vol}(k(N))^{-1}1_{k(N)}(x)\langle x, T\tau \rangle. 	
\end{align*}
The phase $\langle x, T\tau \rangle$ is playing the role of the archimedean phase $e^{ix\xi/h}$ with $h = 1/T$. Note that $a_\tau$ is a $T$-dependent function, but we will often omit this dependence to avoid cumbersome notation.

These functions form the essential building blocks of the test functions we consider in this paper. The operators $\Op(a_\tau)$ have the property that their image consists of vectors that are microlocalized at $\tau$ in the sense of Definition \ref{def:nonarchlocalized}. The main theorem applies to any linear combination of functions $a_\tau$. We will sometimes refer to them as ``non-archimedean wavepackets". If we think of them as being supported on matrices whose elements are in $T^{-1}\mathcal{O}$, then the dual region is given by 
\begin{align*}
\{ \xi \in \g^\ast : |x \cdot \xi | \leq 1 \ \forall x \in T^{-1}k(0)\} =: Tk(0)^\perp.
\end{align*}
\begin{lemma}
The Fourier transform $a_\tau(\xi)$ is a bump function on $T\tau + Tk(0)^\perp$. 	
\end{lemma}
\begin{proof}
	We have
	\begin{align*}
	a_\tau(\xi) = \widehat{a_\tau^\vee}(\xi) &= \frac{1}{\mathrm{vol}(k(N))}\int_{x \in \g} \psi(xT\tau)1_{T^{-1}k(0)}(x)\psi(-x\xi)\,dx \\
	&= \frac{1}{\mathrm{vol}(k(N))}\int_{x \in T^{-1}k(0)} \psi(x(T\tau - \xi))\,dx \\ 
	&= \int_{u \in k(0)} \psi(u(\tau - T^{-1}\xi))\,du 
	= 1_{k(0)^\perp}(\tau - T^{-1}\xi) \\
        &=  1_{T(\tau + k(0)^\perp)}(\xi).
	\end{align*}
As an abuse of notation, we will sometimes write this characteristic function as $1_{T(\tau + \OO)}$ or even more succinctly as just $1_\tau^T$. 
\end{proof}
In other words, $a_\tau(\xi)$ is the characteristic function of the ball $B(T\tau, T)$ where 
\begin{align*}
B(x,R) = \{ y \in \g^\ast : |y - x| \leq R\}. 	
\end{align*}

Then we can imagine tiling the space $\g^\ast$ by translates of $Tk(0)^\perp$, one for each $\tau \in \g^\ast/Tk(0)^\perp$.
This gives rise to the following decomposition (c.f. \cite[Lemma 3.4]{Li}) 
\begin{prop}[Wavepacket Decomposition]\label{prop:wavepacketdecomp}
Let $T =\varpi^{-N}$ and suppose $a$ is a Schwartz function on $\g^\ast$ that is constant on cosets $Tk(0)^\perp$. Let $a_\tau = 1_\tau^T := 1_{T(\tau+\O)}$ be the characteristic function of the ball $B(T\tau, T)$. Then we may write 
\begin{equation}
    a = \sum_{\tau \in \g^\ast / Tk(0)^\perp} a(\tau)1_\tau^T.
\end{equation}
Each term $a1_\tau^T$ is Fourier supported on $T^{-1}k(0)^\perp$ and has constant size on each ball $B(T\tau, T)$. 
\end{prop}

The wavepacket decomposition gives a description of $a$ everywhere on $\g^\ast$, but in practice, the relative character will only ``see" the $a1_\tau^T$ supported near a $T\tau$ that lives within $Tk(0)^\perp$ of the relative coadjoint orbit $\OO_{\pi, \sigma}$. 

\subsection{$\Op$ Calculus} \label{subsec:opcalculus}
We can now define a class of operators 
\begin{align*}
\Op(a) := \int_{x \in \g} a^\vee(x)\pi(\exp x)\,dx. 	
\end{align*}

Fix an element $\tau \in \g^\ast$, once and for all. Then $a_\tau^\vee(x)$ will be the $T$-dependent function as before that oscillates at frequency $T\tau$ on $k(N)$ and $a_\tau(\xi)$ its Fourier transform which is the bump $1_\tau^T$ on $T\tau$. 

We have the Iwahori factorization $K(N) = K_-K_0K_+$, where $K_-$ is the negative unipotent, $K_0$ is the diagonal, and $K_+$ is the unipotent subgroup. We denote the preimages of $K_\star$ under the exponential map by $k_\star$ for $\star \in \{0, +, -\}$. They correspond to the decomposition in the Lie algebra $\g = \g_0 + \g_+ + \g_-$ where each space is given by 
\begin{align*}
    \g_0 := \begin{pmatrix} \ast & \\ & -\ast \end{pmatrix} , \ \g_+ := \begin{pmatrix}  & \ast \\ &  \end{pmatrix}, \ \g_- := \begin{pmatrix} & \\ \ast & \end{pmatrix}.
\end{align*}
The trace pairing induces dualities $\g_0^\ast \cong \g_0$, $\g_{\pm}^\ast \cong \g_{\mp}$. Each $K_\ast$ is one-dimensional, so we can parametrize them by elements of $F$ satisfying appropriate congruence conditions. This means we can decompose $\tau \in \g^\ast$ by $\tau = \tau_0 + \tau_- + \tau_+$ where $\tau_0 \in k(N) \cap \g_0$ and $\tau_{\pm} \in k(N) \cap \g_{\mp}$. For $\star \in \{-, 0, +\}$, we define one-dimensional functions $a_{\tau_\star} \colon F \to \C$ so that
\begin{align*}
a_{\tau_\star}^\vee(x) =  \frac{1}{\mathrm{vol}(k_\star)}\1_{T^{-1}\OO}(x)\psi(xT\tau_\star).	
\end{align*}
Now we can define the operators $\Op^\star$ for $\star \in \{0, \pm\}$ and $a \colon F \to \C$ by 
\begin{align*}
    \Op^\pm(a) = \int_{x \in \g} a^\vee(x) \pi(\exp (x e_{\pm}))\,dx \label{eqn:Opstarb}
\end{align*}
and 
\begin{align*}
    \Op^0(a) = \int_{x \in \g} a^\vee(x) \pi(\exp (x e_0))\,dx
\end{align*}
 where $e_\star$ is the basis element in $k_\star$, for instance $e_+ = \begin{pmatrix} 0 & 1 \\ 0 & 0 \end{pmatrix}$.

On the functions $a_{\tau_\star}$ we have, for instance, 
\begin{align*}
	\Op^\pm(a_{\tau_\mp}) &= \int_{x \in k_\star} a_{\tau_\mp}^\vee(x)\pi(\exp(xe_\pm))\,dx \\
	&= \int_{u \in \OO} \psi(u \tau_\mp) \pi(\exp(T^{-1}ue_\pm))\,du,
\end{align*}
and similarly for $\Op^\pm(a_{\tau_0})$. Therefore, for $\dagger \in \{0, \pm\}$, function $a_{\tau_\dagger}$ has inverse Fourier transform that oscillates in the $\dagger$-direction at frequency $\tau_\dagger$, while $\Op^\star$ averages the action of $\pi$ restricted to $k_\star$. 

The next thing to understand is how these operators compose. We aim for something like 
\begin{align*}
	\Op(a)\Op(b) = \Op(ab),
\end{align*}
so that the operators commute. It will turn out that for our choice of $a$ and $b$, and for our purposes, we can essentially get this property. Before we state the proposition, we define a product
\begin{align*}
    \star : k(n) \times k(n) \to k(n)
\end{align*}
by requiring that
\begin{align*}
e^{x}e^y = e^{x \star y}.	
\end{align*}
The Baker-Campbell-Hausdorff formula then implies that  $x \star y = x + y + \frac{1}{2}[x,y] + O(x^2y, y^2x)$. 
\begin{prop}\label{prop:compositions}
Let $\xi \in k(-m)$, that is each entry of $\xi$ lies in $\varpi^{-m}\OO$. If $N \geq \frac{m}{2}$, then for $x,y \in k(N)$, 
\begin{align}
	\langle x , \xi \rangle \langle y, \xi \rangle = \langle (x \star y), \xi \rangle.
\end{align}	
In other words, the map 
\begin{align}
\psi_\xi : K(N) \xrightarrow[\log]{} k(N) \xrightarrow[]{\langle \ , \ \xi \rangle} \C^\times,	
\end{align}
is a character. 
\end{prop}
\begin{proof}
It suffices to prove that the difference 
\begin{align*}
\{x,y\} = x \star y - x - y 	
\end{align*}
pairs trivially with $\xi$. This is easy to verify. First note that $\{x,y\} = \frac{1}{2}[x,y] + O(x^2y) + O(y^2x)$. For $x,y \in k(N)$, we have $\{x,y\} \in k(2N)$. Therefore, every entry $a_{ij}$ of the matrix $[x,y]$ satisfies $a_{ij} \in \p^{2N}$. So 
\begin{align*}
\langle \{x,y\}, \xi \rangle \equiv \langle 	[x,y], \xi \rangle = \psi(a_{11}\xi_x + a_{12}\xi_y + a_{21}\xi_z - a_{11}\xi_x) = 1
\end{align*}
since each term $a_{ij}\xi_\star \in \OO$ by the condition $2N \geq m$. 
	
\end{proof}

To see how this helps the $\Op$-calculus, define 
\begin{align*}
(a \star b)(\xi) = \int_{x,y} a^\vee(x)b^\vee(y)\langle x \star y , \xi \rangle \,dx\,dy. 	
\end{align*}
 Then by Fourier inversion
\begin{align*}
\Op(a)\Op(b) &= \int_{x,y} a^\vee(x)	b^\vee(y) \pi(e^x)\pi(e^y)\,dx\,dy \\
&= \int_{x,y} a^\vee(x) b^\vee(y) \pi(e^{x \star y}) \,dx\,dy = \Op(a\star b).
\end{align*}
If $x \star y = x + y$, then this would be the standard convolution and we would get $a \star b = \widehat{\left(a^\vee \ast b^\vee\right)} = ab$, and hence $\Op(a \star b) = \Op(ab)$. It suffices to show that the phase $\psi(\cdot \{x,y\}) = 1$ . At this point we will need to make a specific choice of $a$ and $b$. 
\begin{prop}
Recall $T = \varpi^{-N}$ is our scaling parameter and let $a = a_\tau = 1_{T(\tau + \OO)}$ with $|T\tau| \leq q^{2N}$. Then 
\begin{align}\label{eqn:opfactor}
\Op(a_\tau) = \Op^-(a_{\tau_+})\Op^0(a_{\tau_0})\Op^+(a_{\tau_-}),	
\end{align}
and moreover the operators on the right commute.
\end{prop}
\begin{proof}
Recall we have the Iwahori factorization $K(N) = K_-(N)K_0(N)K_+(N)$. Let 
\begin{align*}
    \psi_\tau : K(N) \to \C^\times,
\end{align*}
be the character given by $\psi_\tau(x) = \langle x, T\tau \rangle = \psi(\mathrm{trace}(xT\tau)$. We want to show that that the equation
\begin{align*}
    \psi_\tau(k_-k_0k_+) = \psi_\tau^-(k_-)\psi_\tau^0(k_0)\psi_\tau^+(k_+).
\end{align*}
hold for specific choices of characters $\psi_\tau^{-}$, $\psi_\tau^0$, and $\psi_\tau^+$ of $K_-$, $K_0$, and $K_+$ associated to the coordinates $\tau_-, \tau_+, \tau_0$. 
 This will imply Equation \ref{eqn:opfactor} because we can also factor $\pi(\exp k_-k_0k_+)$ as $\pi(\exp k_-)\pi(\exp k_0)\pi(\exp k_+)$. By Proposition \ref{prop:compositions}, the character $\psi_\tau$ is trivial on $K(2N)$ and moreover $[K(N), K(N)] \subset K(2N)$ so $\psi_\tau$ defines a character on $K(N)/K(2N)$. Therefore, it suffices to show that $\psi_\tau(k_-k_0k_+)$ and $\psi_\tau^-(k_-)\psi_\tau^0(k_0)\psi_\tau^+(k_+)$ agree up to terms in $K(2N)$. To that end, consider the matrix product 
 \begin{align*}
 \begin{pmatrix} 1 & \\ n & 1 \end{pmatrix} \begin{pmatrix} t & \\ & t^{-1} \end{pmatrix} \begin{pmatrix} 1 & u \\ & 1 \end{pmatrix} \in K_-(n)K_0(n)K_+(n).	
 \end{align*}
Expanding we obtain 
\begin{align*}
\begin{pmatrix} t & tu \\ nt & ntu+t^{-1} \end{pmatrix}.
\end{align*}
We know that $t \equiv 1 \bmod \p^N$ and $u,n \equiv 0 \bmod \p^N$ so $tu \equiv u \bmod \p^{2N}$ and $tn \equiv n \bmod \p^{2N}$. Moreover $ntu + t^{-1} \equiv t^{-1} \bmod \p^{2N}$. So this matrix product satisfies 
\begin{align*}
	 \begin{pmatrix} 1 & \\ n & 1 \end{pmatrix} \begin{pmatrix} t & \\ & t^{-1} \end{pmatrix} \begin{pmatrix} 1 & u \\ & 1 \end{pmatrix} \equiv \begin{pmatrix} t & u \\ n & t^{-1} \end{pmatrix} \bmod \p^{2N}.
\end{align*}
It's easy to verify that every ordering of the matrix product above is congruent to \\ $\begin{pmatrix} t & u \\ n & t^{-1} \end{pmatrix} \bmod \p^{2N}$. Since $\psi_\tau$ is trivial on $K(2N)$, this concludes the proof of (\ref{eqn:opfactor}). Next we prove the commutativity. We prove it for $\Op^+$ and $\Op^-$, the other cases are identical. We have
\begin{align*}
    \Op^+(a_{\tau_-})\Op^-(a_{\tau_+}) = \int_{x, y \in \varpi^N \OO} \psi(xT\tau_-)\psi(xT\tau_+)\pi(\exp(ye_+ \star xe_-))\,dx\,dy.
\end{align*}
We can write $ye_+ \star xe_- = ye_+ + xe_- + \{ye_+, xe_-\} = xe_- + ye_+ + \{ye_+,xe_-\}$ and make the change of variables $ze_+ = ye_+ + \{ye_+, xe_-\}$. But note that $ze_+ \equiv ye_+ \bmod k(2N)$, so we obtain 
\begin{align*}
    \Op^+(a_{\tau_-})\Op^-(a_{\tau_+}) &= \int_{x,z \in \varpi^N\OO} \psi(xT\tau_-)\psi(zT\tau_+)\pi(\exp(xe_-))\pi(\exp(ze_+))\,dx\,dz \\
    &= \Op^-(a_{\tau_+})\Op^+(a_{\tau_-}).
\end{align*}
\end{proof}

\section{Invariant Functionals and Relative Characters}\label{sec:functionals}
 Recall from Equation \ref{eqn:relchar} the relative character for our case is given by
\begin{align*}
\mathcal{H}_\sigma(a) = \sum_{v \in \mathcal{B}(\pi)} \ell_\chi(\Op(a)v)\overline{\ell_\chi(v)},	
\end{align*}
for $\pi$ unitary, where 
\begin{align*}
\ell_\chi(v) = \int_{H} W_v\begin{pmatrix} h & \\ & 1 \end{pmatrix} \overline{\chi(h)}\,dh.	
\end{align*}
Here $W_v$ is the image of $v$ in the Whittaker model of $\pi$ and $\sigma$ is one dimensional so it is a character. One can think of the relative character as follows: Take the inner product $\langle g^{-1} \ell_\chi, \ell_\chi \rangle$. This is a distribution on $G$. One can expand the inner product formally, applying Parseval, to obtain the function on $G$
\begin{align*}
\sum_{v \in \mathcal{B}(\pi)} \langle 	\ell_\chi, gv \rangle \langle v, \ell_\chi\rangle.
\end{align*}
Now to obtain a real number out of this, we integrate against $g$ with an integral operator $\pi(f)$ or, equivalently, an operator $\Op(a)$. This gives the relative character. 

In order to make use of this viewpoint, it will be beneficial to identify $\ell_\chi$ with a vector. Such a vector does not exist, in general, but in fact we only need to find a vector that agrees with $\ell_\chi$ on the image of $\Op(a_\tau)$, as that is all we will use the functional for. To this end, we define the following vector:
\begin{definition}
	For $R \in q^{\Z}$, let $v_\chi^R \in \pi$ be a vector in the Kirillov model given by 
	\begin{align}
	v_\chi^R(h) = \chi(h)\1_{[-R, R]}(v(h)).	
	\end{align}
\end{definition}
\begin{prop}
Let $v$ be in $\pi$. Then for $\chi$ a unitary character of $F^\times$, we have
\begin{align}
	\ell_\chi(v) = \lim_{R \to \infty} \langle v, v_\chi^R \rangle. 
\end{align}
\end{prop}
\begin{proof}
 This follows from the behavior of the Whittaker function $W_v$. $W_v(y)$ behaves like $|y|^{1/2-\epsilon}$ as $|y| \to 0$ and $W_v(y)$ decays at infinity (see the discussion in \cite[Sec 3.2]{MV}). Hence $\ell_\chi(v)$ is absolutely convergent and so we can write it as $\lim_{R \to \infty} \langle v, v_\chi^R\rangle$. 
\end{proof}

 \section{Analysis of the Relative Character}\label{sec:main calculations}
 We will now start the analysis of computing the image of $\Op(a_\tau)$ following the proof sketch in Section \ref{subsec:proofsketch}. We won't have to use that $\pi$ is supercuspidal or principal series until the calculation of $\Op^-$ where an analysis of the $\GL_2 \times \GL_1$ functional equation will force a choice of representation $\pi$. The main result of this section is Proposition \ref{prop:table} which computes the relative character. We will prove Proposition \ref{prop:table} in Section \ref{sec:proofofprop}.
 \subsection{The Image of $\Op(a_\tau)$}

We now fix a $\tau$ for the rest of the paper and choose $a = a_\tau = 1_\tau^T =\1_{B(T\tau, T)}(\xi)$. As a short-hand, $a_\star$ will mean $a_{\tau_\star}$. We first start with $\Op^0(a_0)v_\chi^R$. For any vector $v$, we have
\begin{align*}
\ell_\chi(\Op^0(a_0)v) &= \int_{h \in H} \left(\Op^0(a)W_v\right)\begin{pmatrix} h & \\ & 1 \end{pmatrix} \overline{\chi}(h)\,d^\times h	\\
&= \int_{h \in F^\times} \int_{u \in \OO} \psi(u\tau_x)\left(\pi\left(\exp(T^{-1}ue_{0}\right)W_v\right)\begin{pmatrix} h & \\ & 1 \end{pmatrix}\overline{\chi}(h)\,du\,d^\times h \\
&=\int_{h \in F^\times} \overline{\chi}(h)\int_{u \in \OO} \psi(u\tau_x) W_v\begin{pmatrix} h \exp(T^{-1} u) & \\ & 1 \end{pmatrix} \,d^\times h \,d u \\
    &= \int_{u \in \OO} \psi(u\tau_x)\chi(\exp(T^{-1} u))\,du\int_{h \in F^\times} W_v\begin{pmatrix} y & \\ & 1 \end{pmatrix} \overline{\chi}(y)\,d^\times h \\
    &= \left(\int_{u \in \OO} \psi(u \tau_x)\chi(\exp(T^{-1}u))\,du\right)\ell_\chi(v).
\end{align*}
where $e_0$ is the basis element $\begin{pmatrix} 1/2 & \\ & -1/2 \end{pmatrix}$.
It remains to evaluate the integral. Take $N$ large enough so that $2N \geq c(\chi)$, as in the theorem statement. Since $q$ is odd, $\chi(\exp x) = \chi(1+x)$ for $x \in \p^{c(\chi)/2}$. Therefore by Proposition \ref{prop:chi} we can write
\begin{align*}
\ell_\chi(\Op^0(a)v) &= \left(\int_{u \in \OO} \psi(u(\tau_x + \alpha_\chi T^{-1}))\,du\right)\ell_\chi(v) = 1_{\OO}(\tau_x + T^{-1}\alpha_\chi)\ell_\chi(v)\\
&= a_0(\alpha_\chi)\ell_\chi(v)
\end{align*}

We now see what the $H$-direction operator $\Op^0(a_0)$ is really doing. It wants to project onto vectors that are microlocalized at elements in the dual Lie algebra whose rescaled $x$-coordinate is near the pre-image under the map $\mathrm{Lie}(G)^\ast \to \mathrm{Lie}(H)^\ast$ of the conductor $f$ of $\chi$. In other words, it's detecting the `slice' $\OO_\pi \cap \left(\mathrm{c}(\chi) + \mathrm{Lie}(H)^\perp\right)$. \\

We next calculate $\Op^+(a_-)v_\chi^R$. For any vector $W_v$ in the Kirillov model, we have
\begin{align*}
(\Op^+(a_-)W_v)(h)&= \int_{u \in \OO}\psi(u\tau_y)\pi(\exp(T^{-1}ue_{+})W_v)\begin{pmatrix} h & \\ & 1 \end{pmatrix} \,du \\
&= \left(\int_{u \in \OO} \psi(u(\tau_y + T^{-1}h))\,du\right)W_v\begin{pmatrix} h & \\ & 1 \end{pmatrix} 
\end{align*}

The exact behavior of the integral depends on whether $\tau_y \in \OO$, but the general effect is to restrict the support of $h$. Indeed the dichotomy is as follows: 

\begin{align*}
    \int_{u \in \OO} \psi(u(\tau_y + \varpi^N h))\,du = \begin{cases} 1_{v(h) \geq -N}(h) & \text{ if } \tau_y \in \OO \\ 1_{\varpi^{-N-r}(\p^r + \varpi^r\tau_y)}(h) & \text{ if } \tau_y \in \varpi^{-r}\OO^\times, \ r > 0. \end{cases}
\end{align*}

Therefore 
\begin{align*}
\Op^+(a_-)v_\chi^R = \begin{cases}
 v_\chi^R(h)1_{v(h) \geq -N} & \text{ if } \tau_y \in \OO \\
 v_\chi^R(h)1_{\varpi^{-N-r}(\p^r + \varpi^r \tau_y)}(h) & \text{ if } \tau_y \in \varpi^{-r}\OO^\times, \ r > 0.
 \end{cases}
\end{align*}

To understand this better, first we'd like to make sense of the characteristic function $\1_{\varpi^{-N-r}(\p^r + \varpi^r \tau_y)}(h)$. Notice that 
\begin{align*}
\p^r + \varpi^r \tau_y \subset \OO^\times,	
\end{align*}
the characteristic function $\1_{\varpi^{-N-r}(\p^r + \varpi^r \tau_y)}(h)$ defines a character of $\OO^\times$. 
\begin{lemma}
	Let $u_0 \in \OO^\times$ and $\mathcal{X}_m$ denote the group of characters of conductor at most $m$. Then for $u \in \OO^\times$,
	\begin{align}
	1_{u_0 + \p^m}(u)	= \frac{1}{|\mathcal{X}_m|} \sum_{\omega \in \mathcal{X}_m} \omega^{-1}(u_0)\omega(u) =: \E_{\omega \in \mathcal{X}_m} \omega^{-1}(u_0)\omega(u)
	\end{align}
\end{lemma}
\begin{proof}
    The proof follows from Fourier analysis on the finite group $(\OO/\p^m)^\times$. 
\end{proof}

With this, we can update our calculation of $\Op^+(a_+)$ to read 
\begin{align*}
\Op^+(a_-)v_\chi^R = \begin{cases}
 v_\chi^R(h)1_{v(h) \geq -N} & \text{ if } \tau_y \in \OO \\
 v_\chi^R(h)1_{v(h) = -N-r} \E_{\omega \in \mathcal{X}_r} \omega^{-1}(\varpi^r\tau_y)\omega(\varpi^{N+r}h) & \text{ if } \tau_y \in \varpi^{-r}\OO^\times,
\end{cases}
\end{align*}
with $r > 0$. Note that both $\varpi^r \tau_y$ and $\varpi^{N+r}h$ are units. 

\subsection{Negative Unipotent \& Functional Equations}\label{subsec:weyl}

We next need to compute the contribution from the negative unipotent direction, that is, we need to understand $\Op^-(a)v$. First we record a simple lemma. 
\begin{lemma}
    We have 
    \begin{align}
        \ell_\chi(\pi(w)W_v) = \gamma(\pi \otimes \chi)\ell_{\chi^{-1}}(W_v).
    \end{align}
\end{lemma}
\begin{proof}
    This is just a reformulation of the $\GL_2 \times \GL_1$ local functional equation (see Prop \ref{prop:funeq}). Indeed by the local functional equation, the LHS is 
    \begin{align*}
        \int_{h \in F^\times} W_v\left(\begin{pmatrix} h & \\ & 1 \end{pmatrix} w\right)\chi^{-1}(h)\,d^\times h = \gamma(\pi \otimes \chi)\int_{h \in F^\times} W_v\begin{pmatrix} h & \\ & 1 \end{pmatrix} \chi(h)\,d^\times h. 
    \end{align*}
\end{proof}

To compute the negative unipotent case, we will work with the functionals $\ell_\chi$ until it is more convenient to switch to the vectors $v_\chi^R$. For instance, we can start on the functional side by writing for any $v \in \pi$ with $W_v$ in the Kirillov model of $\pi$, 
\begin{align*}
\ell_\chi(\Op^-(a_+)v)	 &= \ell_\chi \left(\int_{u \in \OO} \psi(u\tau_z)(\pi(\exp(T^{-1}ue_{-}))W_v)\,du\right) \\
&= \ell_\chi \left(\int_{u \in \OO} \psi(u\tau_z)(\pi(w)\pi(\exp(T^{-1}ue_{+}))\pi(w) W_v)\,du\right)  \\
&= \gamma(\pi \otimes \chi) \ell_{\overline{\chi}} \left(\int_{u \in \OO} \psi(u\tau_z) \pi(\exp(T^{-1}ue_{+}))\pi(w)W_v\right).
\end{align*}

We should think of the integral inside as doing approximately the same thing as $\Op^+$, but on the vector $\pi(w)W_v$ instead of $W_v$.  However, the resulting dichotomy will depend on whether $\tau_z \in \OO$ rather than $\tau_y$. 

Explicitly,  
\begin{align*}
\ell_{\overline{\chi}} &\left(\int_{u \in \OO} \psi(u\tau_z) \pi(\exp(T^{-1}ue_{+}))\pi(w)W_v\right) \\&= \int_{h \in H}\int_{u \in \OO} \psi(u(\tau_z + T^{-1}h))\,du (\pi(w)W_v)(h)\chi(h)\,d^\times h.
\end{align*}
Therefore for $\tau_z \in \OO$, 
\begin{align*}
\ell_\chi(\Op^-(a_+)v) = \gamma(\pi \otimes \chi) \ell_{\overline{\chi}}\left((\pi(w)v)1_{v(\cdot) \geq -N}\right).	
\end{align*}

When $\tau_z \not\in\O$, the resulting analysis is quite complicated. Therefore, we start with the $\tau_z \in \O$ case, which conveys the central ideas. Since there is a restriction to the support of $\pi(w)v$, we cannot straightforwardly apply the local functional equation again. We need to understand the action of the Weyl element explicitly. To that end, we switch to the vector viewpoint. Namely, 
\begin{align}
\ell_\chi\left(\Op^-(a_+)W_v\right) &= \gamma(\pi \otimes \chi)\ell_{\overline{\chi}}\left(\pi(w)(W_v1_{v(\cdot) \geq -N})\right) \\
&= \gamma(\pi \otimes \chi)\lim_{R \to \infty} \langle \pi(w)(W_v 1_{v(\cdot) \geq -N}), v_\chi^R\rangle \\
&= \gamma(\pi \otimes \chi) \lim_{R \to \infty} \langle W_v, \pi(w)(v_\chi^R 1_{v(\cdot) \geq -N})\rangle. \label{eqn:innerproduct}
\end{align}

In order to complete this calculation, we close this section with a general lemma which encapsulates how the local functional equation can be used to compute the action of $\pi(w)$ on vectors in the Kirillov model. 
\begin{lemma}\label{lemma:weylkirillov}
    Let $[n, \omega] \in \pi$ be given in the Kirillov model by $1_{\varpi^n \OO^\times} \omega$. Then 
    \begin{equation*}
        \pi(w)[n, \omega] = \omega^{-1}(y)\frac{\log q}{2\pi i}\int_{|z| = 1} \gamma(\pi \otimes \omega^{-1}|\cdot|^s)z^{v(y) + n}\frac{\,dz}{z},
    \end{equation*}
    where $z = q^{-s}$. 
\end{lemma}

\begin{proof}
    We proceed by Mellin inversion. We can write 
    \begin{align}
        (\pi(w)[n, \omega])(y) &= \int_{\widehat{F^\times}} \gamma(\pi \otimes \nu) \nu(y)\int_{F^\times} [n, \omega](h)\nu(h)\,d^\times h\,d\nu \\
        &= \int_{\widehat{F^\times}} \gamma(\pi \otimes \nu)\nu(y)\int_{F^\times} 1_{\varpi^n \OO^\times}(h)\omega(h)\nu(h)\,d^\times h \,d\nu \\
        &= \int_{\widehat{F^\times}} \gamma(\pi \otimes \nu)\nu(y)\omega\nu(\varpi^n)\int_{u \in \OO^\times} \omega(u)\nu(u)\,d^\times u \,d\nu.
    \end{align}
    The second integral detects the condition that $\omega = \nu^{-1}$ on $\OO^\times$, or in other words, $\omega\nu$ is unramified. Therefore, we can write $\omega\nu = |\cdot|^s$ for some $s \in \C$. If $x = \varpi^k u$ with $u \in \OO^\times$, then $|x|^s = q^{-ks}$. Writing $z = q^{-s}$, we can therefore reparametrize the above integral as follows: 
    \begin{align*}
        (\pi(w)[n, \omega])(y) &= \omega^{-1}(y)\frac{\log q}{2\pi i} \int_{|z| = 1} \gamma(\pi \otimes \nu)  z^{v(y) + n}\frac{\,dz}{z}.
    \end{align*}
    Writing $\nu$ as $\omega^{-1}\nu\omega$ and using that $\nu\omega$ is unramified gives the result. 
\end{proof}

The next two sections will be spent computing $\pi(w)(v_\chi^R1_{v(\cdot) \geq -N})$ in the Kirillov model by specializing Lemma \ref{lemma:weylkirillov} to the case where $\pi$ is supercuspidal or principal series, and then applying Cauchy integral formula. 

\subsection{Supercuspidal Case}
Recall we would like to calculate Equation \ref{eqn:innerproduct}, i.e. the right hand side of
\begin{equation*}
   \ell_\chi\left(\Op^-(a_+)W_v\right)  = \gamma(\pi \otimes \chi) \lim_{R \to \infty} \langle W_v, \pi(w)(v_\chi^R 1_{v(\cdot) \geq -N})\rangle. 
\end{equation*}
To calculate Equation \ref{eqn:innerproduct} we need to understand $\pi(w)(v_\chi^R1_{v(\cdot) \geq -N})$ in the Kirillov model. To do that, we will first calculate $\pi(w)(v_\chi^R1_{\varpi^m\OO^\times})$ and sum the result over a range. We now need to use that we are working with a specific representation in order to pin down the gamma factor, and will carry out the resulting analysis separately for the supercuspidal case and the principal series case.  Recall that under the Local Langlands Correspondence there is a representation $\rho = \Ind_{E/F}\xi := \Ind_{\mathcal{W}_E}^{\mathcal{W}_F} \xi $ corresponding to $\pi$. This correspondence is compatible with gamma and epsilon factors. For any character $\eta$ of $F^\times$ or $\psi$ of $F$, let $\eta_E := \eta \circ \mathrm{Nm}_{E/F}$ and $\psi_E := \psi \circ \mathrm{Tr}_{E/F}$. Let $f := f(E|F)$ be the inertial degree. 

\begin{prop}\label{prop:weylonshell}
In the Kirillov model, for $R$ sufficiently large we have
\begin{align}
\pi(w)(v_\chi^R1_{\varpi^m\OO^\times})(y) = \epsilon(1/2, \xi \chi_E^{-1}, \psi_E) v^R_{\chi^{-1}}(y) 1_{v(y) = -m-fc(\xi\chi_E^{-1})}.
\end{align}

\end{prop}

\begin{cor}
Then 
\begin{align}
\pi(w)(v_\chi^R1_{v(\cdot) \geq -N}) = \epsilon(1/2, \xi \chi_E^{-1}, \psi_E) v^R_{\chi^{-1}} 1_{v(\cdot) \leq N-fc(\xi\chi_E^{-1})}
\end{align}
	
\end{cor}

\begin{proof}[Proof of Proposition \ref{prop:weylonshell}]
	First, we note that for $R$ sufficiently large,  $v_\chi^R(h)1_{\varpi^m\OO^\times}(h) = \chi(h)1_{\varpi^m \OO^\times}$. Therefore, by Lemma \ref{lemma:weylkirillov} 
\begin{align*}
\pi(w)(v_\chi^R(y)1_{v(y) = m}) = \frac{\log q}{2\pi i }\int_{|z| = 1} \gamma(\pi \otimes \nu)\overline{\chi(y)}z^{v(y) + m} \frac{\,dz}{z}.
\end{align*}
Since $\pi$ is supercuspidal, the $L$-factors are equal to $1$, so $\gamma(\pi \otimes \nu) = \epsilon(\pi \otimes \nu)$. Then 
\begin{align*}
\epsilon(1/2, \pi \otimes \nu, \psi) = \epsilon(1/2, \xi \nu_E, \psi_E).
\end{align*}
Since $\chi\nu$ is unramified, $\chi_E\nu_E$ is unramified, and for $x = \varpi^k u \in F^\times$, $\chi_E\nu_E(x) = z^{fk}$,  where $f$ is the inertial degree $f(E|F)$. So
\begin{align*}
\epsilon(1/2, \xi\nu_E, \psi_E) = \epsilon(1/2, \xi\chi_E^{-1}\chi_E\nu_E, \psi_E) = \epsilon(1/2, \xi\chi_E^{-1}, \psi_E)z^{fc(\xi\chi_E^{-1})}.	
\end{align*}

Therefore, 
\begin{align*}
\pi(w)(v_\chi^R(y)1_{v(y) = m}) &= \epsilon(1/2, \xi\chi_E^{-1}, \psi_E)\frac{\log q}{2\pi i} \int_{|z|=1} z^{fc(\xi\chi_E^{-1}+m+v(y)}\frac{\,dz}{z} \\
&= \epsilon(1/2, \xi\chi_E^{-1}, \psi_E)1_{v(y) = -m-fc(\xi\chi_E^{-1})}.	
\end{align*}
\end{proof}

When $\tau_z \in \OO$, we can write 
\begin{align*}
\ell_\chi(\Op^-(a_+)W_v) &= \gamma(\pi \otimes \chi)\lim_{R \to \infty} \langle W_v, \pi(w)(v_\chi^R 1_{v(\cdot) \geq -N}) \rangle \\	
&= \epsilon(1/2, \xi\chi_E, \psi_E) \lim_{R \to \infty}\langle W_v, \epsilon(1/2, \xi \chi_E^{-1}, \psi_E) v^R_{\chi^{-1}}(y) 1_{v(y) \leq N-fc(\xi\chi_E^{-1})}\rangle \\
&= \epsilon(\xi\chi_E)\overline{\epsilon(\xi\chi_E^{-1})} \ell_{\overline{\chi}}\left(W_v 1_{v(\cdot) \leq N - fc\left(\xi\chi_E^{-1}\right)}\right).
\end{align*}

This completes the case where $\tau_z \in \OO$. The case where $\tau_z \not\in \OO$ is not conceptually more difficult, only notationally cumbersome. Recall that in this case we are calculating 
\begin{align*}
\ell_\chi(\Op^-(a_+)W_v) = \gamma(\pi \otimes \chi)\lim_{R \to \infty} \langle W_v, \pi(w)(v_\chi^R 1_{v(\cdot) = -N - s} \E_{\omega \in \mathcal{X}_s} \omega^{-1}(\varpi^s \tau_z)\omega(\varpi^{N+s})\omega) \rangle
\end{align*}
for $\tau_z \in \varpi^{-s}\OO^\times$ with $s$ positive. To do this, we again use  Mellin inversion. 

\begin{prop}\label{prop:nonintegerweyl}
	In the Kirillov model, for $R$ sufficiently large, we have
	\begin{align}
	 \E_{\omega \in \mathcal{X}_s} \pi(w)(v_\chi^R 1_{v(\cdot) = -N - s}\omega) \omega^{-1}(\varpi^s \tau_z)\omega(\varpi^{N+s})
	\end{align}
is equal to 
\begin{align}
v^R_{\chi^{-1}}1_{v(\cdot) = N+s - fc(\xi\chi_E^{-1}\omega_E^{-1})}\E_{\omega \in \mathcal{X}_s} \epsilon(\xi\chi_E^{-1}\omega_E^{-1})\omega^{-1}(\tau_z) \omega(\varpi^N)\omega^{-1} . 	
\end{align}

\end{prop}
\begin{proof}
	As in the first proof, we first remark that for $R$ large, $v_\chi^R(h)1_{\varpi^m \OO^\times}(h) = \chi(h)1_{\varpi^m\OO^\times}$ for any fixed $m$.   
By Lemma \ref{lemma:weylkirillov}, we can write this all as 
\begin{align*}
	\E_{\omega \in \mathcal{X}_s} \omega^{-1}(\tau_z) \omega(\varpi^N) \frac{\log q}{2\pi i }\int_{|z|=1} \gamma(\pi \otimes \nu) \chi^{-1}(y)\omega^{-1}(y)z^{v(y) -N - s}\frac{\,dz}{z}.
\end{align*}

As before, the gamma factor can be simplified to 
\begin{align*}
\gamma(\pi \otimes \nu) &= \epsilon(1/2, \xi\nu_E, \psi_E) = \epsilon(1/2, \xi\chi_E^{-1}\omega_E^{-1}\chi_E\omega_E\nu_E, \psi_E) \\
&= \epsilon(1/2, \xi\chi_{E}^{-1}\omega_E^{-1}, \psi_E)z^{fc(\xi\chi_E^{-1}\omega_{E}^{-1})}
\end{align*}
where $f := f(E|F)$ is the inertial degree. So our expression becomes
\begin{align*}
	&\E_{\omega \in \mathcal{X}_s} \omega^{-1}(\tau_z) \omega(\varpi^N) \frac{\log q}{2\pi i }\int_{|z|=1}\epsilon(\xi\chi_E^{-1}\omega_E^{-1}) \chi^{-1}(y)\omega^{-1}(y)z^{v(y) -N - s+fc(\xi\chi_E^{-1}\omega_E^{-1})}\frac{\,dz}{z} \\
	&= \E_{\omega \in \mathcal{X}_s} \epsilon(\xi\chi_E^{-1}\omega_E^{-1}) \omega^{-1}(\tau_z) \omega(\varpi^N) \chi^{-1}(y)\omega^{-1}(y) 1_{v(y) = N+s - fc(\xi\chi_E^{-1}\omega_E^{-1})}.
\end{align*}
For $R$ large, $v^R_{\chi^{-1}} 1_{v(y) = N+s - fc(\xi\chi_E^{-1}\omega_E^{-1})}$ agrees with $\chi^{-1}(y)1_{v(y) = N+s - fc(\xi\chi_E^{-1}\omega_E^{-1})}$, which yields the result. 
\end{proof}

\subsection{Principal Series Case}
As a reminder, in analogy to our work in the supercuspidal case our aim is to compute $\pi(w)(v_\chi^R 1_{v(\cdot) \geq -N})$ in the Kirillov model. To do this, we will compute $\pi(w)(v_\chi^R 1_{v(\cdot) = m})$ and then sum the result over $m \geq -N$. Recall that $\pi = \chi_0 \boxplus \chi_0^{-1}$, where $\chi_0$ is some ramified character of $F^\times$. We will assume that neither $\chi_0\chi^{-1}$ nor $\chi_0^{-1}\chi^{-1}$ is unramified, in which case we say $(\chi_0, \chi^{-1})$ is a generic pair. We will denote the characters $\chi_0\chi^{-1}$ and $\chi_0^{-1}\chi^{-1}$ by $\chi_\flat$ and $\chi_\sharp$ with conductors $f_\flat$ and $f_\sharp$. 

\begin{prop}\label{prop:principal series}
    Let $(\chi_0, \chi^{-1})$ be a generic pair. Then in the Kirillov model for $R$ sufficiently large we have, 
    \begin{equation*}
    \pi(w)(v_\chi^R(y)(y)1_{v(y) = m}) = \epsilon(1/2, \chi_\flat)\epsilon(1/2, \chi_\sharp) v^R_{\chi^{-1}}(y)1_{v(y) = -m - f_\sharp - f_\flat}.
\end{equation*}
\end{prop}
\begin{cor}
When $\tau_z \in \OO$, then we must have
$$\pi(w)(v_\chi^R(y) 1_{v(y) \geq -N}) = \epsilon(1/2, \chi_b)\epsilon(1/2, \chi_\sharp)v^R_{\chi^{-1}}(y) 1_{v(y) \leq N - f_\sharp - f_\flat}.$$
\end{cor}
\begin{proof}[Proof of Proposition \ref{prop:principal series}]
First note that for $R$ sufficiently large, $$v_\chi^R(h)1_{\varpi^m \OO^\times}(h) = \chi(h)1_{\varpi^m\OO^\times}.$$ By Lemma \ref{lemma:weylkirillov},  we can write this as 
\begin{align*}
(\pi(w)v_\chi^R)(y)1_{v(y) = m} = \frac{\log q}{2\pi i }\int_{|z| = 1} \gamma(\pi \otimes \nu) \overline{\chi}(y)z^{v(y)+m}\frac{\,dz}{z}.	
\end{align*}

Next we use the factorization of the gamma factor as 
\begin{align*}
    \gamma(1/2, \pi \otimes \nu) &= \epsilon(1/2, \pi \otimes \nu)\frac{L(1/2, \tilde{\pi}\otimes \nu^{-1})}{L(1/2, \pi \otimes \nu)} \\ 
    &= \epsilon(1/2, \chi_0 \nu)\epsilon(1/2, \chi_0^{-1}\nu)\frac{L(1/2,\chi_0^{-1}\nu^{-1})L(1/2, \chi_0\nu^{-1})}{L(1/2, \chi_0\nu)L(1/2, \chi_0^{-1}\nu)}.
\end{align*}

Recall $\chi_\flat = \chi_0 \chi^{-1}$ and $\chi_\sharp = \chi_0^{-1}\chi^{-1}$ with conductor exponents equal to $f_\flat$ and $f_\sharp$ respectively. Then we can write the epsilon factors as as $\epsilon(1/2, \chi_\flat z^{\rm{val}}) = \epsilon(1/2, \chi_\flat)z^{f_\flat}$ and $\epsilon(1/2, \chi_\sharp z^{\rm{val}}) =\epsilon(1/2, \chi_\sharp)z^{f_\sharp}$.

When $(\chi_0, \chi)$ is a generic pair, the $L$-factors are all ramified and equal to $1$. So the gamma factor is 
\begin{equation}
    \gamma(1/2, \pi \otimes \nu) = \epsilon(1/2, \chi_\flat)\epsilon(1/2, \chi_\sharp)z^{f_\flat + f_\sharp}.
\end{equation}
Then we can calculate as follows: 
\begin{align}
    \pi(w)(v_\chi^R(y)1_{v(y) = m}) &= \epsilon(1/2, \chi_\flat)\epsilon(1/2, \chi_\sharp)\overline{\chi}(y) \frac{\log q}{2\pi i }\int_{|z| = 1} z^{f_\flat + f_\sharp + m + v(y)}\frac{\,dz}{z} \\
    &= \epsilon(1/2, \chi_\flat)\epsilon(1/2, \chi_\sharp) \overline{\chi}(y)1_{v(y) = -m - f_\sharp - f_\flat}.
\end{align}
\end{proof}

Therefore, when $\tau_z \in \OO$, we can write 
\begin{align*}
\ell_\chi(\Op^-(a_+)W_v) &= \gamma(\pi \otimes \chi)\lim_{R \to \infty} \langle W_v, \pi(w)(v_\chi^R1_{v(\cdot) \geq -N})\rangle \\
&= \gamma(\pi \otimes \chi)\overline{\epsilon(1/2, \chi_\flat)\epsilon(1/2, \chi_\sharp)}\ell_{\chi^{-1}}(W_v1_{v(\cdot) \leq N - f_\sharp - f_\flat}).
\end{align*}
This completes the case where $\tau_z \in \OO$. 
Next we need to understand what happens when $\tau_z \not \in \OO$. 
\begin{prop}
Let $\tau_z \in \varpi^{-s}\OO^\times$ be such that $s < \min(f_\sharp, f_\flat)$. Then in the Kirillov model, for $R$ sufficiently large, we have
\begin{align}
\pi(w)(v_\chi^R1_{v(\cdot) = -N -s} \mathbb{E}_{\omega \in \mathcal{X}_s} \omega^{-1}(\varpi^s \tau_z)\omega(\varpi^{N+s})\omega)
\end{align}
is equal to 
\begin{align}
	v^R_{\chi^{-1}}1_{v(\cdot) = N+s - f_\flat^\omega - f_\sharp^\omega}\mathbb{E}_{\omega \in \mathcal{X}_s} \omega^{-1}(\varpi^s \tau_{z})\omega^{-1}\epsilon_{f_\flat^\omega}\epsilon_{f_\sharp^\omega}.
\end{align}
\end{prop}
\begin{proof}
	Using Lemma \ref{lemma:weylkirillov}, we have
\begin{align*}
    &\pi(w)(v_\chi^R(y) \mathbb{E}_{c(\omega) \leq s} 1_{v(y) = -N-s}\overline{\omega}(\varpi^s\tau_{z})\omega(y)) \\
    &= \mathbb{E}_{c(\omega) \leq s} \overline{\omega(\varpi^s \tau_{z})} \frac{\log |q|}{2\pi i}\int_{|z|=1} \gamma(\pi \otimes \omega^{-1}\chi^{-1}|\cdot|^{s'})\omega^{-1}(y)\chi^{-1}(y)z^{v(y) - N - s}\frac{\,dz}{z}  \\
    &= \mathbb{E}_{c(\omega) \leq s} \overline{\omega(\varpi^s \tau_{z})} \epsilon(\chi_0\omega^{-1}\chi^{-1})\epsilon(\chi_0^{-1}\omega^{-1}\chi^{-1})\omega^{-1}\chi^{-1}(y)I(\flat, \sharp, N) \\
    &= \mathbb{E}_{c(\omega) \leq s} \omega^{-1}(\varpi^s \tau_{z})\chi^{-1}(y)\omega^{-1}(y) \epsilon(\chi_\flat\omega^{-1})\epsilon(\chi_\sharp \omega^{-1})1_{v(y) = N+s - f_\flat^\omega - f_\sharp^\omega},
\end{align*}
where $\epsilon_{f_\flat^\omega}$ and $\epsilon_{f_\sharp^\omega}$ are the epsilon factors of the characters $\chi_\flat \omega$ and $\chi_\sharp \omega$ and $$I(\flat, \sharp, N) = \frac{\log |q|}{2\pi i}\int_{|z|=1} z^{f_\flat^\omega + f_\sharp^\omega +v(y) - N - s}\frac{\,dz}{z}.$$
\end{proof}

\begin{rem}
	In the above proof, the crucial assumption, $s < \min(f_\sharp, f_\flat)$ ensures that twists by $\omega$ don't introduce $L$-factors so that we can write the gamma factor as 
\begin{equation}
    \gamma(1/2, \pi \otimes \nu) = \epsilon(1/2, \chi_\flat \omega^{-1}) \epsilon(1/2, \chi_\sharp \omega^{-1}) z^{f_\flat^\omega + f_\sharp^\omega},
\end{equation}
where $f_\flat^\omega$ and $f_\sharp^\omega$ are the conductors of $\chi_\flat \omega^{-1}$ and $\chi_\sharp \omega^{-1}$ respectively. This assumption should be automatically satisfied in the relevant ranges since the conductor of $\omega$ is small compared to the conductors of $\chi_\flat$ and $\chi_\sharp$. 
\end{rem}

\subsection{Main Proposition}
Recall, once again, that the strategy to calculate the relative character is to factor it using the Iwahori factorization as

\begin{align*}
\langle \Op(a)v_\chi^R, v_\chi^R \rangle = \langle \Op^+(a_+)\pi(w)v_\chi^R, \pi(w)\Op^0(a_0)\Op^+(a_-)v_\chi^R \rangle,	
\end{align*}
and compute this as $R \to \infty$. 

By the local functional equation, this can be written as 
\begin{align*}
\gamma(\pi \otimes \overline{\chi})a_0(\alpha_\chi)\langle \Op^+(a_+)v^R_{\chi^{-1}}, \pi(w)\Op^+(a_-)v_\chi^R\rangle. 	
\end{align*}
 We then case on the positions of $\tau_y$ and $\tau_z$. Recall when $\tau_z \not\in\OO$, we write $\tau_z \in \varpi^{-s}\OO^\times$ with $s > 0$ and when $\tau_z \in \OO$ we take $s = 0$, and similarly for $\tau_y$ and $r$. Recall the characters $\chi_\sharp = \chi_0^{-1}\chi^{-1}$ and $\chi_\flat = \chi_0\chi^{-1}$. We have the following proposition:
\begin{prop}\label{prop:table}
    Keep the assumptions of Theorem \ref{thm:main theorem}. Then the value of the inner product $\gamma(\pi\otimes \chi^{-1})\langle \Op^+(a_{\tau_+})v^R_{\chi^{-1}}, \pi(w)\Op^+(a_{\tau_-})v_\chi^R\rangle$ is given by the following table:
    \begin{table}[h]
\centering
\begin{tabular}{c|c|c}
\multicolumn{1}{c}{} & \multicolumn{2}{c}{$\tau_y$} \\
\cline{2-3}
$\tau_z$ & $\tau_y \in \mathcal{O}$ & $\tau_y \not\in \mathcal{O}$ \\
\hline
$\tau_z \in \mathcal{O}$ & 
\begin{minipage}{0.35\textwidth}
\vspace{0.5cm}
\centering
$\left(2N - c(\pi, \chi) \right)\1_{2N \geq c(\pi, \chi)}$
\vspace{0.5cm}
\end{minipage} & 
\begin{minipage}{0.35\textwidth}
\vspace{0.5cm}
\centering
$\frac{1}{|\mathcal{X}_r|}\1_{2N+r \geq c(\pi, \chi)}$
\vspace{0.5cm}
\end{minipage} \\
\hline
$\tau_z \not\in \mathcal{O}$ & 
\begin{minipage}{0.35\textwidth}
\vspace{0.5cm}
\centering
$\frac{1}{|\mathcal{X}_s|}\1_{2N+s \geq c(\pi, \chi)}$
\vspace{0.5cm}
\end{minipage} & 
\begin{minipage}{0.35\textwidth}
\vspace{0.5cm}
\centering
$\frac{1}{|\mathcal{X}_{\max(r,s)}|} \1_{U(min(r,s))}\left(\frac{\alpha_{\pi, \chi}}{T\tau_yT\tau_z}\right)$
\vspace{0.5cm}
\end{minipage} \\
\end{tabular}
\caption{Table of $\tau_y$ and $\tau_z$ relationships}
\label{tab:tau_relationships}
\end{table}
where $c(\pi, \chi) = f_\sharp + f_\flat$ in the principal series case and $fc(\xi\chi_E^{-1})$ in the supercuspidal case and $\alpha_{\pi, \chi} \in \varpi^{-c(\pi, \chi)}\OO^\times$ is equal to  $\alpha_{\chi_\sharp} \alpha_{\chi_\flat}$ in the principal series case and $\mathrm{Nm}(\alpha_{\xi\chi_E^{-1}})$ in the supercuspidal case.
\end{prop}
\section{Proof of Proposition \ref{prop:table}}\label{sec:proofofprop}
\subsection{Supercuspidal Case}
\begin{proof}
We start with the $\tau_y, \tau_z \in \OO$ case. Then 
\begin{align*}
&\langle \Op^+(a_-)v^R_{\chi^{-1}}, \pi(w)\Op^+(a_+)v_\chi^R \rangle = \langle \Op^+(a_-)v^R_{\chi^{-1}}, \pi(w)v_\chi^R 1_{v(\cdot) \geq -N} \rangle \\
&= \langle v^R_{\chi^{-1}}1_{v(\cdot) \geq -N}, \epsilon(1/2, \xi \chi_E^{-1}, \psi_E) v^R_{\chi^{-1}}(y) 1_{v(y) \leq N-fc(\xi\chi_E^{-1})}\rangle \\
&= \overline{\epsilon(1/2, \xi\chi_E^{-1}, \psi_E)}\left(2N - fc(\xi\chi_{E}^{-1})\right)
\end{align*}

The next case is when $\tau_y \in \OO$ and $\tau_z \in \varpi^{-s}\OO^\times$ for some $s > 0$. Then 
\begin{align*}
&\langle \Op^+(a_-)v^R_{\chi^{-1}}, \pi(w)\Op^+(a_+)v_\chi^R \rangle = \langle \Op^+(a_-)v^R_{\chi^{-1}}, \epsilon(\xi\chi_E^{-1})v^R_{\chi^{-1}}1_{v(\cdot) \leq N - fc(\xi\chi_E^{-1})}\rangle	\\
&= \langle v^R_{\chi^{-1}} 1_{v(\cdot) = -N-s}\E_{\omega \in \mathcal{X}_s}\omega^{-1}(\varpi^s \tau_z)\omega(\varpi^{N+s})\omega, v^R_{\chi^{-1}}\epsilon(\xi\chi_E^{-1})1_{v(\cdot) \leq N - fc(\xi\chi_E^{-1})}\rangle.
\end{align*}

By linearity this inner product is the same as 
\begin{align*}
	&\left(\overline{\epsilon(\xi\chi_E^{-1})}\E_{\omega \in \mathcal{X}_s}\omega^{-1}(\varpi^s\tau_z)\omega(\varpi^{N+s})\right)\\
    &\times\int_{F^\times} 1_{[-R, R]}\overline{\chi}(y)\omega(y)\chi(y)1_{v(y) = -N-s}1_{v(y) \leq N - fc(\xi\chi_E^{-1})}\,d^\times y\\
	&=\overline{\epsilon(\xi\chi_E^{-1})}\E_{\omega \in \mathcal{X}_s} \omega^{-1}(\varpi^s \tau_z)1_{\omega = |\cdot|^{s'}}1_{-N-s \leq N - fc(\xi\chi_E^{-1})},
\end{align*}	
for some $s' \in \C$ since $\int_{\OO^\times} \omega(u)\,d^\times u$ is only nonzero if $\omega$ is unramified. Therefore, we have calculated that in this case 
\begin{align*}
\langle \Op^+(a_-)v^R_{\chi^{-1}}, \pi(w)\Op^+(a_+)v_\chi^R \rangle = \frac{\epsilon(\xi\chi_E^{-1})}{|\mathcal{X}_s|}1_{2N+s \geq fc(\xi\chi_E^{-1})}.	
\end{align*}

Next case is when $\tau_y \in \varpi^{-r}\OO^\times$ for some $r > 0$ and $\tau_z \in \OO$. This case is similar. We have 
\begin{equation*}
    \langle \Op^+(a_-)v^R_{\chi^{-1}}, \pi(w)\Op^+(a_+)v_\chi^R \rangle.
\end{equation*}
The right-hand side of the inner product is given by 
\begin{equation*}
    \pi(w)v_\chi^R  1_{v(\cdot) = -N-r}\E_{\omega \in \mathcal{X}_r} \omega^{-1}(\varpi^r \tau_y) \omega(\varpi^{N+r}) \omega
\end{equation*}
while the left-hand side is just $v^R_{\chi^{-1}}1_{v(\cdot) \geq -N}$. Therefore, the inner product is equal to 
\begin{align*}
	&\langle  v^R_{\chi^{-1}}1_{v(\cdot) \geq -N}, v^R_{\chi^{-1}} 1_{v(\cdot) = N+r - fc(\xi\chi_E^{-1}\omega_{E}^{-1})} \E_{\omega \in \mathcal{X}_r} \epsilon(\xi\chi_E^{-1}\omega_E^{-1}) \omega^{-1}(\tau_y)\omega(\varpi^N)\omega\rangle \\
	&= \frac{\epsilon(\xi\chi_E^{-1})}{|\mathcal{X}_r|}1_{2N+r \geq fc(\xi \chi_E^{-1})},
\end{align*}
again since the inner product picks out $\omega = 1$. 

The last case is when $\tau_y \in \varpi^{-r}\OO^\times$ and $\tau_z \in \varpi^{-s}\OO^\times$, with $r, s > 0$. This gives
\begin{equation*}
    \langle \Op^+(a_-)v^R_{\chi^{-1}}, \pi(w)\Op^+(a_+)v_\chi^R \rangle.
\end{equation*}
The left-hand side of this inner product is 
\begin{equation*}
    v^R_{\chi^{-1}} 1_{v(\cdot) = -N-s}\E_{\omega \in \mathcal{X}_s}\omega^{-1}(\varpi^s \tau_z)\omega(\varpi^{N+s})\omega
\end{equation*}
while the right-hand side is 
\begin{equation*}
    v^R_{\chi^{-1}} 1_{v(\cdot) = N+r - fc(\xi\chi_E^{-1}\eta_{E}^{-1})} \E_{\eta \in \mathcal{X}_r} \epsilon(\xi\chi_E^{-1}\eta_E^{-1}) \eta^{-1}(\varpi^r\tau_y)\eta(\varpi^{N+r})\eta.
\end{equation*}
Combining the two gives 
\begin{align*}
	& \E_{\omega \in \mathcal{X}_s}\E_{\eta \in \mathcal{X}_r}\overline{\epsilon(\xi\chi_E^{-1}\eta_E^{-1})}\omega^{-1}(\varpi^s \tau_z)\eta(\varpi^r \tau_y)I(\flat, \sharp, N, r, s) \\
	&= \E_{\omega \in \mathcal{X}_s}\E_{\eta \in \mathcal{X}_r} \overline{\epsilon(\xi\chi_E^{-1}\eta_E^{-1})} \eta(\varpi^{s+r}\tau_y\tau_z)1_{2N + r + s = fc(\xi\chi_E^{-1}\eta_E^{-1})}\1_{\omega^{-1}=\eta},
\end{align*}
where
\begin{equation*}
    I(\flat, \sharp, N, r, s) = \int_{F^\times}\omega(y) \overline{\eta}(y)1_{v(y) = -N-s}1_{v(y) = N+r - fc(\xi\chi_E^{-1}\eta_E^{-1})}\,d^\times y
\end{equation*}

We need to understand how the epsilon factor $\epsilon(1/2, \xi\chi_E^{-1}, \psi_E)$ behaves under twists of characters $\eta_E$ with small conductor. How small is the conductor? The character $\eta$ has conductor at most $r$, and $r < fc(\xi\chi_E^{-1}\eta_E^{-1}) - 2N$ since $s > 0$. For $N$ large enough, the characteristic function $1_{2N + r + s = fc(\xi\chi_E^{-1}\eta_E^{-1})}$ is zero, so the assumption that $r$ is much smaller than the conductor is not very strict. Therefore, we assume $r \leq c(\xi\chi_E^{-1})/2$. This gives $c(\xi\chi_E^{-1}\eta_E^{-1}) = c(\xi\chi_E^{-1})$. Furthermore, let $\alpha_{\xi\chi_E^{-1}} \in \varpi^{-c(\xi\chi_E^{-1})}\OO^\times$ such that $\xi\chi_E^{-1}(1+x) = \psi(\alpha_{\xi\chi_E^{-1}}x)$ for $x \in \p^{c(\xi\chi_E^{-1})/2}$. Now since $c(\eta_E) \leq r \leq c(\xi\chi_E^{-1})/2$, by Proposition \ref{prop:tate}
\begin{align*}
\epsilon(1/2, \xi\chi_E^{-1}\eta_E^{-1}, \psi_E) = \eta_E^{-1}(\alpha_{\xi\chi_E^{-1}})\epsilon(1/2, \xi\chi_E^{-1}, \psi_E).
\end{align*}
Note that 
\begin{align*}
\eta_E^{-1}(\alpha_{\xi\chi_E^{-1}}) = \eta^{-1}(\alpha_{\pi, \chi}).	
\end{align*}
The condition $r \leq c(\xi\chi_E^{-1})/2$ in light of $\omega^{-1} = \eta$ forces $s \leq c(\xi\chi_E^{-1})/2$. Then from the characteristic function, we have
\begin{align*}
fc(\xi\chi_E^{-1}) = 2N + r + s \leq 2N + c(\xi\chi_E^{-1}).	
\end{align*}

Since $E$ is a quadratic extension $f \in \{1,2\}$. If $f = 1$ we obtain $N \geq 0$ and if $f = 2$, we obtain the "uncertainty principle" $N \geq \frac{c(\xi\chi_E^{-1})}{2}$.  So now our inner product is given by 
\begin{align*}
\overline{\epsilon(\xi\chi_E^{-1})} \E_{\omega \in \mathcal{X}_s} \E_{\eta \in \mathcal{X}_r}\1_{\eta = \omega^{-1}} \eta(\varpi^{s+r}\tau_y\tau_z\alpha_{\pi, \chi})1_{2N+r+s = fc(\xi\chi_E^{-1})}.
\end{align*}

Suppose without loss of generality that $s < r$. The above can then be written
\begin{align*}
	\overline{\epsilon(\xi\chi_E^{-1})} \E_{\eta \in \mathcal{X}_r} \E_{\omega \in \mathcal{X}_s} 1_{\eta = \omega^{-1}}\omega^{-1}(\varpi^{s+r}\tau_y\tau_z\alpha_{\pi, \chi})1_{2N+r+s = fc(\xi\chi_E^{-1})}.
\end{align*} 
The average over $\omega$ forces the argument to lie in $U(s)$, and the average over $\eta$ contributes $\frac{1}{|\mathcal{X}_r|}$, which yields the desired result.
\end{proof}

\subsection{Principal Series Case}
\begin{proof}
We start with the $\tau_y, \tau_z \in \OO$. Then 
\begin{align*}
\langle \Op^+(a_-)v^R_{\chi^{-1}}, \pi(w)\Op^+(a_+)v_\chi^R \rangle &= \langle \Op^+(a_-)v^R_{\chi^{-1}}, \pi(w)(v_\chi^R 1_{v(\cdot) \geq -N}) \rangle \\
&= \langle v^R_{\chi^{-1}}1_{v(\cdot) \geq -N}, \epsilon_\sharp \epsilon_\flat v^R_{\chi^{-1}}1_{v(\cdot) \leq N - f_\flat - f_\sharp}\rangle \\
&= \overline{\epsilon(1/2, \chi_\sharp)\epsilon(1/2, \chi_\flat)}(2N - f_\flat - f_\sharp)\1_{2N\geq f_\sharp + f_\flat}.
\end{align*}

The next case is when $\tau_y \in \OO$ and $\tau_z \in \varpi^{-s}\OO^\times$ for some $s > 0$. Then 
\begin{align*}
&\langle \Op^+(a_-)v^R_{\chi^{-1}}, \pi(w)\Op^+(a_+)v_\chi^R\rangle = \langle \Op^+(a_-)v^R_{\chi^{-1}}, \epsilon_\sharp\epsilon_\flat v^R_{\chi^{-1}}1_{v(\cdot) \leq N - f_\sharp - f_\flat}\rangle \\
&= \langle  v^R_{\chi^{-1}}1_{v(\cdot) = -N-s} \E_{\omega \in \mathcal{X}_s} \omega^{-1}(\varpi^r\tau_y)\omega(\varpi^{N+s})\omega , \epsilon_\sharp\epsilon_\flat v^R_{\chi^{-1}}1_{v(\cdot) \leq N - f_\sharp - f_\flat}\rangle.
\end{align*}
By linearity, this inner product is the same as 	
\begin{align*}
	\overline{\epsilon_\sharp\epsilon_\flat}\E_{\omega \in \mathcal{X}_s}\omega^{-1}(\varpi^s\tau_z)\omega(\varpi^{N+s})&\int_{F^\times} 1_{[-R, R]}\overline{\chi}(y)\omega(y)\chi(y)1_{v(y) = -N-s}1_{v(y) \leq N - f_\sharp - f_\flat}\,d^\times y\\
	&=\overline{\epsilon_\sharp \epsilon_\flat}\E_{\omega \in \mathcal{X}_s} \omega^{-1}(\varpi^s \tau_z)1_{\omega = |\cdot|^{s'}}1_{-N-s \leq N - f_\flat - f_\sharp},
\end{align*}	
for some $s' \in \C$ since $\int_{\OO^\times} \omega(u)\,d^\times u$ is only nonzero if $\omega$ is unramified. Therefore, we have calculated that in this case 
\begin{align*}
\langle \Op^+(a_-)v^R_{\chi^{-1}}, \pi(w)\Op^+(a_+)v_\chi^R \rangle = \frac{\overline{\epsilon(1/2, \chi_\sharp)\epsilon(1/2, \chi_\flat)}}{|\mathcal{X}_s|}1_{2N+s \geq f_\flat + f_\sharp}.	
\end{align*}

Next case is when $\tau_y \in \varpi^{-r}\OO^\times$ for some $r > 0$ and $\tau_z \in \OO$. This case is very similar. We have 
\begin{equation*}
    \langle \Op^+(a_-)v^R_{\chi^{-1}}, \pi(w)\Op^+(a_+)v_\chi^R \rangle.
\end{equation*}
The right-hand side of the inner product is given by 
\begin{equation*}
    \pi(w)v_\chi^R  1_{v(\cdot) = -N-r}\E_{\omega \in \mathcal{X}_r} \omega^{-1}(\varpi^r \tau_y) \omega(\varpi^{N+r}) \omega,
\end{equation*}
while the left-hand side of the inner product is just $v^R_{\chi^{-1}}1_{v(\cdot) \geq -N}$. Therefore the inner product evaluates to 
\begin{align*}
	&\langle  v^R_{\chi^{-1}}1_{v(\cdot) \geq -N}, v^R_{\chi^{-1}} 1_{v(\cdot) = N+r - f_\sharp - f_\flat} \E_{\omega \in \mathcal{X}_r} \epsilon_\flat \epsilon_\sharp \omega^{-1}(\tau_y)\omega(\varpi^N)\omega\rangle \\
	&= \frac{\overline{\epsilon(1/2, \chi_\flat)\epsilon(1/2, \chi_\sharp)}}{|\mathcal{X}_r|}1_{2N+r \geq f_\flat + f_\sharp},
\end{align*}
again since the inner product picks out $\omega = 1$. 

Finally, we tackle the case when $\tau_y \in \varpi^{-r}\OO^\times$ and $\tau_z \in \varpi^{-s}\OO^\times$, with $r,s > 0$. This gives
\begin{equation*}
    \langle \Op^+(a_-)v^R_{\chi^{-1}}, \pi(w)\Op^+(a_+)v_\chi^R \rangle.
\end{equation*}
The right-hand side of the inner product is given by 
\begin{equation*}
    v^R_{\chi^{-1}}1_{v(y) = N+r - f_\flat^\eta - f_\sharp^\eta}\mathbb{E}_{\eta \in \mathcal{X}_r} \eta^{-1}(\varpi^s \tau_{ij})\eta^{-1}(y) \epsilon(\chi_\flat\eta^{-1})\epsilon(\chi_\sharp\eta^{-1})
\end{equation*}
while the left-hand side is given by 
\begin{equation*}
    v^R_{\chi^{-1}}1_{v(\cdot) = -N - s} \mathbb{E}_{\omega \in \mathcal{X}_s} \omega^{-1}(\varpi^s\tau_z)\omega(\varpi^{N+s})\omega.
\end{equation*}
Combining the two to compute the full inner product yields 
\begin{align*}
& \mathbb{E}_{\omega \in \mathcal{X}_s} \mathbb{E}_{\eta \in \mathcal{X}_r} \overline{\epsilon(\chi_\flat\eta^{-1})\epsilon(\chi_\sharp\eta^{-1})} \omega^{-1}(\varpi^s\tau_z)\eta(\varpi^r\tau_y)\ I(\flat, \sharp, N, r) \\
&= \mathbb{E}_{\omega \in \mathcal{X}_s} \mathbb{E}_{\eta \in \mathcal{X}_r} \overline{\epsilon(\chi_\flat\eta^{-1})\epsilon(\chi_\sharp\eta^{-1})} \1_{2N + r + s = f_\flat + f_\sharp} \1_{\omega^{-1} = \eta},
\end{align*}
where 
\begin{equation*}
    I(\flat, \sharp, N, r) = \int_{\F^\times} \omega(y)\overline{\eta}(y)\1_{v(y) = -N - r}\1_{v(y) = N + r - f_\sharp - f_\flat}\,d^\times y
\end{equation*}

We now need to understand how the epsilon factors $\epsilon(1/2, \chi_\flat \eta^{-1})$ and $\epsilon(1/2, \chi_\sharp \eta^{-1})$ behave under twists by characters $\eta$ with small conductor. Since $\eta$ has conductor at most $r$, and $s > 0$, we have $r < f_\sharp + f_\flat - 2N$. For $N$ large enough, the characteristic function $\1_{2N + r + s = f_\flat + f_\sharp}$ is zero, so we can and do assume $r$ is much smaller than $f_\flat$ and $f_\sharp$, without losing much. Therefore, $c(\chi_\sharp \eta^{-1}) = c(\chi_\sharp)$ and $c(\chi_\flat \eta^{-1}) = c(\chi_\flat)$. Let $\alpha_\flat := \alpha_{\chi_\flat}$ and $\alpha_\sharp := \alpha_{\chi_\sharp}$. Since $c(\eta) \leq r \leq \min(f_\flat/2, f_\sharp/2)$, we have
\begin{align*}
\epsilon(1/2, \chi_\flat \eta^{-1}, \psi) &= \eta^{-1}(\alpha_\flat)\epsilon(1/2, \chi_\flat, \psi)	\\
\epsilon(1/2, \chi_\sharp \eta^{-1}, \psi) &= \eta^{-1}(\alpha_\sharp)\epsilon(1/2, \chi_\sharp, \psi).
\end{align*}
The condition $r \leq \min(f_\flat/2, f_\sharp/2)$ yields $s \leq \min(f_\flat/2, f_\sharp/2)$ by symmetry. So we have
\begin{align*}
f_\flat + f_\sharp = 2N + r + s \leq 2N + 2\min(f_\flat/2,f_\sharp/2).
\end{align*}
This implies ``uncertainty principle" $N \geq \max(f_\flat/2, f_\sharp/2)$. So the inner product is given by 
\begin{align*}
\overline{\epsilon(1/2, \chi_\flat)\epsilon(1/2, \chi_\sharp)}\E_{\omega \in \mathcal{X}_s}\E_{\eta \in \mathcal{X}_r} \1_{\eta = \omega^{-1}} \eta(\varpi^{s+r}\tau_y\tau_z \alpha_\sharp^{-1}\alpha_\flat^{-1})\1_{2N + r + s = f_\flat + f_\sharp}.
\end{align*}
Suppose without loss of generality that $s < r$. The above can then be written as 
\begin{align*}
\overline{\epsilon(1/2, \chi_\flat)\epsilon(1/2, \chi_\sharp)}\E_{\omega \in \mathcal{X}_s}\E_{\eta \in \mathcal{X}_r}\1_{\eta = \omega^{-1}} \omega^{-1}(\varpi^{s+r}\tau_y\tau_z \alpha_\sharp^{-1}\alpha_\flat^{-1})\1_{2N + r + s = f_\flat + f_\sharp}.
\end{align*}
The average over $\omega$ forces the argument of $\omega^{-1}$ to lie in $U(s)$ and the average over $\eta$ contributes $\frac{1}{|\mathcal{X}_r|}$, which yields the desired result. 
\end{proof}

\section{Phase Space Integral}\label{sec:integral}
 We will now complete the proof of the main theorem \ref{thm:main theorem}. We'd like to relate our calculations of the relative characters to integrals over ``phase space" of our functions $a$. In other words, we plan to complete the proof of the equality
\begin{align*}
\langle \Op(a)v_\chi, v_\chi \rangle = \int_{\mathrm{Hyp}(\pi, \chi)} a(\xi)\,d\xi	
\end{align*}
by calculating the right hand side. It suffices to take $a = a_\tau = 1_\tau^T$ by Proposition \ref{prop:wavepacketdecomp}. The map 
\begin{align*}
\mathrm{Hyp} : \{\text{ Pairs of irreps of } (G,H)	\} &\to \mathrm{Lie}(G)^\ast \\
(\pi, \sigma) & \mapsto \mathrm{Hyp}(\pi, \sigma),
\end{align*}
is meant to map the data of $\pi$ and $\sigma = \chi$ to the hyperbola that we integrate over. For us, we have
\begin{align*}
\mathrm{Hyp}(\pi, \chi) = \{(x,y,z) \ : \ x = \alpha_\chi, \ yz = \alpha_{\pi, \chi}\},
\end{align*}
where $c(\pi, \chi) = f_\flat + f_\sharp$ if $\pi$ is principal series and $c(\pi, \chi) = fc(\xi\chi_E^{-1})$ if $\pi$ is supercuspidal. We will take the measure $\,d\xi$ to be the multiplicative Haar measure on $F^\times$.
Now we integrate the characteristic function 
\begin{align*}
a_\tau(\xi)  = \1_{B(T\tau, T)}(\xi)	
\end{align*}
 of the box of size $T$ around the rescaled microlocal frequency $T\tau$. Like before, we first notice that if $T\tau_x$ must be within $T$ of $\alpha_\chi$, otherwise the whole integral is zero (the support of the characteristic function won't overlap $\mathrm{Hyp}(\pi, \chi)$ in that case. Let's again case on when $\tau_y, \tau_z \in \OO$. 

If $\tau_y, \tau_z \in \OO$, then for instance, $\xi_y - T\tau_y \in T\OO \implies T^{-1}\xi_y \in \OO$. Similarly $T^{-1}\xi_z \in \OO$. So we can write $\xi_y = T\xi'_y$ and $\xi_z = T\xi_z'$ with $\xi'_y, \xi'_z \in \OO$. Then  
\begin{align*}
    \xi_y \xi_z = \alpha_{\pi, \chi} \implies T\xi'_yT\xi'_z = \alpha_{\pi, \chi} \implies \xi'_y\xi'_z = \varpi^{2N }\alpha_{\pi, \chi}
\end{align*}
So we have
\begin{align*}
1 \geq |\xi_y'| \geq q^{-2N +c(\pi, \chi)},	
\end{align*}
since $|\xi_z| \leq 1$. So we must have $2N \geq c(\pi, \chi)$, otherwise there is no such $\xi_y$. Hence, the integral becomes 
\begin{align*}
\int_{q^{-2N +c(\pi, \chi)}}^1 \frac{\,d\xi_y'}{|\xi_y'|} = 2N - c(\pi, \chi), 
\end{align*}
and our contribution in total is 
\begin{align*}
(2N - c(\pi, \chi))\times \1_{2N \geq c(\pi, \chi)}.	
\end{align*}

We next do the case when $\tau_y, \tau_z \in F \setminus \OO$. In this case, we have $T^{-1}\xi_y \in \tau_y + \OO$ and $T^{-1}\xi_z \in \tau_z + \OO$. Let $|\tau_y| = q^r$ and $|\tau_z| = q^s$. Then we can write $\xi_y = T\tau_yu_y$ and $\xi_z = T\tau_zu_z$ with $u_y \in U(r)$ and $u_z \in U(s)$. Then we have

\begin{align*}
\xi_y\xi_z = \alpha_{\pi, \chi} \implies u_yu_z = \frac{\varpi^{2N}}{\tau_y\tau_z}\alpha_{\pi, \chi} = \frac{\alpha_{\pi, \chi}}{T\tau_yT\tau_z}
\end{align*}

Now suppose $s < r$. Then the right-hand side of the above equation lies in $U(s)$. The stricter congruence is then $u_y \in U(r)$ and for each choice there is exactly one $u_z \in U(s)$ satisfying the above equation. The integral is scaling invariant, because we use the measure $\frac{\,dy}{y}$, so in this range we need to compute
\begin{align*}
\int_{u_yu_z = \frac{\varpi^{2N-c(\pi, \chi)}}{\tau_y\tau_z\alpha_{\pi, \chi}}} \,d^\times u_y = \mathrm{vol}(U(r)) = \frac{1}{q^r(1-q^{-1})}. 
\end{align*}

In total we get 
\begin{align*}
\frac{1}{q^r(1-q^{-1})} \times \1_{U(s)}\left(\frac{\varpi^{2N}}{\tau_y\tau_z}\alpha_{\pi, \chi}.\right)	
\end{align*}

The next case is when $\tau_y \in \F \setminus \OO$ and $\tau_z \in \OO$. This is just the $s = 0$ version of the previous case. In this case $U(s) = U(0) = \OO$ by convention, and this gives rise to the condition that $2N +r - c(\pi, \chi) \geq 0$. So in this case the contribution to the integral is
\begin{align*}
\frac{1}{q^r(1-q^{-1})} \times \1_{2N + r \geq c(\pi, \chi)}.	
\end{align*}

The final case is when $\tau_y \in \OO$ and $\tau_z \in F\setminus\OO$. This case is completely symmetric except now $r = 0$ and $s > r$. So by the exact same argument we get
\begin{align*}
\frac{1}{q^s(1-q^{-1})}\times \1_{2N + s \geq c(\pi, \chi)}.	
\end{align*}

\end{document}